\documentclass{amsart}
\usepackage{amsfonts,color}
\usepackage{latexsym}
\usepackage{amssymb}
\usepackage{amsmath}

\usepackage{amsthm}
\usepackage{epsfig}
\usepackage{graphicx,colortbl}
\usepackage{url}
\usepackage[colorlinks=false, pdfborder={0 0 0}]{hyperref}
\usepackage{mathtools, color, changes}

\usepackage{mathptmx}

\usepackage{relsize}

\usepackage{titleps}


\newcommand{\R}{\mathbb R}
\newcommand{\N}{\mathbb N}
\newcommand{\Z}{\mathbb Z}

\newcommand{\norma}[1]{\left\Vert#1\right\Vert}

\newtheorem{thm}{Theorem}[section]
\newtheorem{cor}{Corollary}[section]
\newtheorem{lemma}{Lemma}[section]

\theoremstyle{remark}
\newtheorem{rmk}{Remark}

\theoremstyle{definition}

\def\vol{\mathrm{vol}}
\numberwithin{equation}{section}

\begin{document}


\title{A discrete approach to the functional Zhang's projection inequality}

\author[D.\,Alonso]{David Alonso-Guti\'errez}
\address{\'Area de an\'alisis matem\'atico, Departamento de matem\'aticas, Facultad de Ciencias, Universidad de Zaragoza, Pedro Cerbuna 12, 50009 Zaragoza (Spain), IUMA}
\email[(David Alonso)]{alonsod@unizar.es}

\author[J. S\'anchez]{Julia S\'anchez-Loscertales}
\address{\'Area de an\'alisis matem\'atico, Departamento de matem\'aticas, Facultad de Ciencias, Universidad de Zaragoza, Pedro Cerbuna 12, 50009 Zaragoza (Spain), IUMA}
\email[(Julia S\'anchez)]{julia.sanchez@unizar.es}

\thanks{The authors are partially supported by Project PID2022-137294NB-I00  funded by   MICIU/AEI/10.13039/501100011033/ and by DGA Project E48\_23R}

\begin{abstract}
In this paper, we focus our attention on the following inequality, which holds for any integrable log-concave function with positive integral $f$, and gives an inclusion between the $n$-th Ball body of the covariogram function of $f$ and its polar projection body, providing a functional version of Zhang's projection inequality in the setting of log-concave functions:
$$
n\int_0^\infty r^{n-1}\int_{\R^n}\min\{f(z),f(z-re_n)\}dzdr\leq n!\frac{\left(\int_{\R^n}f(x)dx\right)^{n+1}}{\left(\int_{e_n^\perp}P_{e_n^\perp}f(y)dy\right)^n}.
$$
Here $(e_i)_{i=1}^n$ denotes the canonical basis in $\R^n$  and $P_{e_n^\perp} f$ denotes the projection of $f$ onto the hyperplane orthogonal to $e_n$ given, for $y\in e_n^\perp$, by $P_{e_n^\perp} f(y)=\sup_{\lambda\in\R}f(y+\lambda e_n)$. We will prove a discrete version of this inequality in which we will consider discrete measures involving the lattice point enumerator measure, rather than the Lebesgue measure. In order to prove it, we will obtain a discrete version of the functional version of Berwald's inequality. We will also provide a discrete version of the functional Rogers-Shephard inequality. All the discrete versions of the inequalities considered will imply their continuous counterparts.
\end{abstract}

\date{\today}
\maketitle

\section{Introduction and notation}

Given a convex body $K\subseteq\R^n$, i.e., a compact convex set with non-empty interior, it is well known that, as a consequence of Cauchy's projection formula \cite[Eq. (A.45)]{G}, the quantity given by
$$
\Vert x\Vert_{\Pi^*(K)}:=\Vert x\Vert_2|P_{x^\perp}(K)|_{n-1},\quad x\in\R^n,
$$
defines a norm in $\R^n$. Here, and along the paper, $\Vert\cdot\Vert_2$ denotes the Euclidean norm, whose closed unit ball in $\R^n$ will be denoted by $B_2^n$,  $P_{x^\perp}(K)$ denotes the orthogonal projection of $K$ onto the linear hyperplane orthogonal to $x$, and $|A|_k$ (or simply $|\cdot|$ when the dimension is clear from the context) denotes the $k$-dimensional volume (Lebesgue measure) of a set $A$ contained in a $k$-dimensional affine subspace. The closed unit ball of this norm is a convex body, called the polar projection body of $K$, which will be denoted by $\Pi^*(K)$.

As a direct consequence of \cite[Thm. 4.1.5]{G}, for any convex body $K\subseteq\R^n$, the quantity $|K|^{n-1}|\Pi^*(K)|$ is an affine invariant, i.e., for any non-degenerate affine map $T$ we have that $$|T(K)|^{n-1}|\Pi^*(T(K))|=|K|^{n-1}|\Pi^*(K)|.$$ The well-known Petty projection inequality, which is stronger than the isoperimetric inequality, was proved in \cite{Pe} and it states that, among all $n$-dimensional convex bodies, the affinely invariant quantity $|K|^{n-1}|\Pi^*K|$ is maximized if and only if $K$ is an ellipsoid. Therefore, for every convex body $K\subseteq\R^n$,
\begin{equation}\label{eq:Petty}
|K|^{n-1}|\Pi^*(K)|\leq\left(\frac{|B_2^n|}{\left|B_2^{n-1}\right|}\right)^n.
\end{equation}

In \cite{Zh}, Zhang proved a reverse Petty projection inequality, showing that among all $n$-dimensional convex bodies, the same affinely invariant quantity is minimized if and only if $K$ is a simplex (i.e., the convex hull of $n+1$ affinely independent points). Thus, for any convex body $K\subseteq\R^n$

\begin{equation}\label{eq:Zhang}
\frac{{{2n}\choose{n}}}{n^n}\leq|K|^{n-1}|\Pi^*(K)|.
\end{equation}

In the last decades, many geometric parameters and inequalities have been extended from the setting of convex bodies to the setting of log-concave functions, i.e., functions $f:\R^n\to[0,\infty)$ of the form $f(x)=e^{-u(x)}$, with $u:\R^n\to (-\infty,\infty]$ a convex function (see, for instance \cite{AKV} and \cite{Ba0}, where the Blaschke-Santal\'o inequality is extended to the setting of log-concave integrable functions, or \cite{KM}, where reverse forms of the Brunn-Minkowski and Blaschke-Santal\'o inequalities for log-concave functions were proved). In particular, denoting $\mathcal{F}(\R^n)$ the set of integrable (with respect to the Lebesgue measure) log-concave functions defined on $\R^n$ with positive integral, the polar projection body of $f\in\mathcal{F}(\R^n)$, denoted by $\Pi^*(f)$, was defined in \cite{AGJV} as the closed unit ball of the norm given by
$$
\Vert x\Vert_{\Pi^*(f)}:=2\Vert x\Vert_2\int_{x^\perp}P_{x^\perp}f(y)dy,
$$
where $P_{x^\perp}f:x^\perp\to[0,\infty)$ denotes the projection of $f$ onto $x^\perp$, which was defined in \cite{KM} as $\displaystyle{P_{x^\perp}f(y):=\sup_{\lambda\in\R}f(y+\lambda x)}$. Let us point out that the factor 2 in the definition of $\Vert\cdot\Vert_{\Pi^*(f)}$ was introduced so that, whenever $f$ belongs to the Sobolev space $W^{1,1}(\R^n)=\left\{f\in L^1(\R^n)\,:\,\frac{\partial f}{\partial x_i}\in L^1(\R^n),\,\forall 1\leq i\leq n\right\}$, the norm $\Vert f\Vert_{\Pi^*(f)}$ coincides with the norm given by
\begin{equation}\label{eq:NormP^*(f)2}
\Vert x\Vert_{\Pi^*(f)}=\int_{\R^n}|\langle\nabla f(y),x\rangle|dy,
\end{equation}
which was considered in \cite{Zh2} in order to prove the functional version of \eqref{eq:Petty}, which states that for every $f\in W^{1,1}(\R^n)$, denoting by $\Pi^*(f)$ the unit ball of the norm given by \eqref{eq:NormP^*(f)2}, we have
$$
\Vert f\Vert_{\frac{n}{n-1}}|\Pi^*(f)|^{1/n}\leq\frac{|B_2^n|}{2\left|B_2^{n-1}\right|}.
$$
Inequality \eqref{eq:Petty} can be recovered from the latter inequality by, given a convex body $K\subseteq\R^n$, approximating $\chi_K$, the characteristic function of $K$, by smooth functions in $W^{1,1}(\R^n)$.

In \cite{ABG1}, the following extension of Zhang's inequality to the setting of log-concave functions was proved: For any $f\in\mathcal{F}(\R^n)$ we have
\begin{equation}\label{eq:FunctionalZhang}
\int_{\R^n}\int_{\R^n}\min\{f(x),f(y)\}dydx\leq 2^n n!\Vert f\Vert_1^{n+1}|\Pi^*(f)|.
\end{equation}
Given a convex body $K\subseteq\R^n$ containing the origin, applying inequality \eqref{eq:FunctionalZhang} to the log-concave function given by $f(x)=e^{-\Vert x\Vert_K}$, where $\Vert\cdot\Vert_K$ denotes the Minkowski functional of $K$, defined as $\Vert x\Vert_K:=\inf\{\lambda\geq0\,:\,x\in\lambda K\}$ the geometric form of Zhang's inequality, \eqref{eq:Zhang}, is recovered.

Let us point out that in \cite{GZ}, a proof of inequality \eqref{eq:Zhang} was obtained as a consequence of an inclusion relation between the $n$-th Ball body of the covariogram function of $K$ and the convex body $\Pi^*(K)$ and that the proof of \eqref{eq:FunctionalZhang} provided in \cite{ABG1}, also follows from a similar inclusion relation. Let us introduce some notation in order to state these inclusion relations.

Given a convex body $K\subseteq\R^n$, its covariogram function is the function $g_K:\R^n\to[0,\infty)$, supported on $K-K$, the Minkowski sum of $K$ and its opposite $-K$, defined as
$$
g_K(x):=|K\cap(x+K)|.
$$
Given $f\in\mathcal{F}(\R^n)$, its covariogram function is the function $g_f:\R^n\to[0,\infty)$ defined as
$$
g_f(x):=\int_{\R^n}\min\left\{\frac{f(z)}{\Vert f\Vert_\infty},\frac{f(z-x)}{\Vert f\Vert_\infty}\right\}dy.
$$
Notice that if $K\subseteq\R^n$ is a convex body, then $g_{\chi_K}=g_K$.

Given $g:\R^n\to[0,\infty)$ a measurable function such that $g(0)>0$, the following family of sets associated to $g$ were introduced by Ball in \cite{Ba}. For any $p>0$ the $p$-th Ball body associated to $g$ is the convex body defined as
$$
K_p(g):=\left\{x\in\R^n\,:\,\int_0^\infty pr^{p-1}g(rx)dr\geq g(0)\right\}.
$$
They are star sets with $0$ as a center whose radial function (which, for any star set $K$ with $0$ as a center, is defined for every $\theta\in S^{n-1}$, as $\rho_K(\theta)=\sup\{\lambda\geq 0\,:\,\lambda\theta\in K\}$) is, for every $\theta\in S^{n-1}$,
$$
\rho_{K_p(g)}(\theta)=\left(\frac{p}{g(0)}\int_0^\infty r^{p-1}g(r\theta)dr\right)^{1/p}.
$$
Moreover, if $g\in\mathcal{F}(\R^n)$ then $K_p(g)$ is a convex body for every $p>0$. Taking into account that for any convex body $K\subseteq\R^n$ and any $f\in\mathcal{F}(\R^n)$ we have that $g_K, g_f\in\mathcal{F}(\R^n)$ and $g_K(0), g_f(0)>0$, the families of convex bodies $K_p(g_K)$ and $K_p(g_f)$ are defined for every convex body $K\subseteq\R^n$ and every $f\in\mathcal{F}(\R^n)$.

In \cite{GZ}, Zhang's inequality was obtained  as a consequence of the following inclusion relation: For every convex body $K\subseteq\R^n$,
\begin{equation}\label{eq:InclusionZhang}
{\binom{2n}{n}}^\frac{1}{n}K_n(g_K)\subseteq n|K|\Pi^*(K).
\end{equation}
Its functional version, given by equation \eqref{eq:FunctionalZhang}, was proved in \cite{ABG1} (see also \cite{ABG2}) by showing that for every $f\in\mathcal{F}(\R^n)$
\begin{equation}\label{eq:InclusionFunctionalZhang}
K_n(g_f)\subseteq 2(n!)^\frac{1}{n}\Vert f\Vert_1\Pi^*(f).
\end{equation}

In \cite{IYNZ}, the authors obtained a Brunn-Minkowski type inequality for the discrete measure given by the lattice point enumerator $G_n(A)=\sharp(A\cap\Z^n)$, where $\sharp (\cdot)$ denotes the cardinality of a set (see Theorem \ref{thm: BM_lattice_point_no_G(K)G(L)>0} below, where $C_n$ denotes the open cube $(-1,1)^n$. More generally, if $A$ is contained an affine subspace $x_0+\text{span } \{e_1, \ldots, e_k \}$ for some $x_0 \in \text{span } \{e_1, \ldots, e_k \}^{\perp}$ with $1\leq k\leq n$, we will denote $G_k(A)=G_n(A-x_0)$. We will also denote $C_k=(-1,1)^k\times\{0\}^{n-k}$ and $B_\infty^n=[-1,1]^n$). Since then, several discrete versions of geometric and functional inequalities have been obtained following classical ideas in convex geometry. Typically, an extra open unit cube appears in some terms in these new discrete inequalities, due to the fact that adding an open unit cube in the large term in the discrete version of Brunn-Minkowski inequality is necessary (see for instance \cite{IYNZ}, where besides the discrete version of Brunn-Minkowski inequality, a discrete version of Borell-Brascamp-Lieb inequality was proved, \cite{HLN}, where a discrete version of the $L_p$-Brunn Minkowski inequality was obtained, or \cite{AG}, where a discrete version of Borell's inequality was proved).  Nevertheless, such discrete versions allow to recover their geometric counterparts.

In \cite{ALM}, it was pointed out that having inclusion \eqref{eq:InclusionZhang} for every convex body $K\subseteq\R^n$ is equivalent to having the following inequality between the radial functions of $K_n(g_K)$ and of $\Pi^*(K)$, for every convex body $K\subseteq\R^n$, just in the direction $e_n$, being $(e_i)_{i=1}^n$ the canonical basis in $\R^n$:
\begin{equation}\label{eq:InclusionZhangRadialFunctions1}
{\binom{2n}{n}}^\frac{1}{n}\rho_{K_n(g_K)}(e_n)\leq n|K|\rho_{\Pi^*(K)}(e_n).
\end{equation}
Therefore, having the inclusion given by \eqref{eq:InclusionZhang} for every convex body $K\subseteq\R^n$ is equivalent to having that, for every convex body $K\subseteq\R^n$,
\begin{equation}\label{eq:InclusionZhangRadialFunctions2}
\frac{{\binom{2n}{n}}}{n^n}n\int_0^\infty r^{n-1}|K\cap(re_n+K)|dr\leq \frac{|K|^{n+1}}{|P_{e_n^\perp}(K)|^n}.
\end{equation}
After noticing this fact, the following discrete version of inequality \eqref{eq:InclusionZhangRadialFunctions2} was proved for convex bodies satisfying that $\displaystyle{\max_{y\in e_n^\perp}|K\cap (y+\langle e_n\rangle)|_1=|K\cap \langle e_n\rangle|_1}$:
\begin{equation}\label{eq:InequalityDiscreteZhang}
\frac{{\binom{2n}{n}}}{n^n}n\int_0^\infty r^{n-1}\mu(K\cap(re_n+K))dr\leq \frac{(\mu(S_{e_n}(K)+C_{n-1})^{n+1}}{G_{n-1}(P_{e_n^\perp}(K)))^n},
\end{equation}
where $\langle e_n\rangle$ denotes the 1-dimensional linear space spanned by $e_n$, $d\mu$ is the measure defined by $d\mu(x)=d\mu_n(x):=dG_{n-1}\otimes dm_1$, being $dm_k$ the $k$-dimensional Lebesgue measure, and $S_{e_n}(K)$ denotes the Steiner symmetrization of $K$ with respect to the hyperplane $e_n^\perp$.

\begin{rmk}
Let us point out that, besides ensuring that $G_{n-1}(P_{e_n^\perp}(K))>0$, the condition $\displaystyle{\max_{y\in e_n^\perp}|K\cap (y+\langle e_n\rangle)|_1=|K\cap \langle e_n\rangle|_1}$  on the convex body $K$ for \eqref{eq:InequalityDiscreteZhang} to hold appears because it was proved by applying the discrete version of Berwald's inequality (see Theorem \ref{thm:BerwaldDiscreteNew} below, obtained in \cite[Thm. 4.5]{ALY} with the extra assumption that the function attained its maximum at the origin) to the concave function $f: P_{e_n^\perp}K\to[0,\infty)$ given by $f(y)=|K\cap (y+\langle e_n\rangle)|_1$. However, we will show later in this paper that Theorem \ref{thm:BerwaldDiscreteNew} can be obtained without that extra assumption. Therefore, inequality \eqref{eq:InequalityDiscreteZhang} holds for every convex body $K\subseteq\R^n$ such that $0\in P_{e_n^\perp}(K)$.
\end{rmk}

Notice that, in the same way, having the inclusion relation \eqref{eq:InclusionFunctionalZhang} for every $f\in\mathcal{F}(\R^n)$, is equivalent to having the corresponding inequality between radial functions in the direction $e_n$, for every $f\in\mathcal{F}(\R^n)$, which reads as
\begin{equation}\label{eq:InclusionFunctionalZhangRadialFunctions}
n\int_0^\infty r^{n-1}\int_{\R^n}\min\{f(z),f(z-re_n)\}dzdr\leq n!\frac{\left(\int_{\R^n}f(x)dx\right)^{n+1}}{\left(\int_{e_n^\perp}P_{e_n^\perp}f(y)dy\right)^n}.
\end{equation}

The main purpose of this paper is to prove the following theorem, which gives a discrete version of \eqref{eq:InclusionFunctionalZhangRadialFunctions} in the same spirit as inequality \eqref{eq:InequalityDiscreteZhang} gives a discrete version of inequality \eqref{eq:InclusionZhangRadialFunctions2}. Given $f\in\mathcal{F}(\R^n)$, we will denote by $S_{e_n}(f)$ the Steiner symmetrization of $f$ with respect to the hyperplane $e_n^\perp$ (see the precise definition in Section \ref{sec:Steiner}) and, for any  $f,g\in\mathcal{F}(\R^n)$, we will denote $f\star g$ their Asplund product given by
$$
f\star g(x)=\sup_{z\in\R^n}f(z)g(x-z).
$$

With this notation, we state the theorem as follows:

\begin{thm}\label{thm:DiscreteFunctionalZhang}
Let $f\in\mathcal{F}(\R^n)$ such that $\Vert f\Vert_\infty=f(0)$. Then,
$$
n\int_0^\infty r^{n-1}\int_{\R^n}\min\{f(z),f(z-re_n)\}d\mu(z)dr\leq n!\frac{\left(\int_{\R^n}S_{e_n}(f)\star\chi_{C_{n-1}}(x)d\mu(x)\right)^{n+1}}{\left(\int_{e_n^\perp}P_{e_n^\perp}f(y)dG_{n-1}(y)\right)^n}.
$$
\end{thm}

\begin{rmk}
We will show that Theorem \ref{thm:DiscreteFunctionalZhang} implies inequality \eqref{eq:InclusionFunctionalZhangRadialFunctions} and, therefore the inclusion given by \eqref{eq:InclusionFunctionalZhang} and the functional version of Zhang's inequality given by \eqref{eq:FunctionalZhang}. However, since the geometric versions of Zhang's inequality \eqref{eq:Zhang} and the inclusion \eqref{eq:InclusionZhang} are recovered from the functional versions (inequality \eqref{eq:FunctionalZhang} and inclusion \eqref{eq:InclusionFunctionalZhang}) by considering, given a convex body $K\subseteq\R^n$ with $0\in K$, functions of the form $f(x)=e^{-\Vert x\Vert_K}$ rather than $\chi_K$, it is not the case that we recover inequality \eqref{eq:InequalityDiscreteZhang} from Theorem \ref{thm:DiscreteFunctionalZhang}.
\end{rmk}

Another classical inequality in convex geometry, which is intimately related to the properties of the covariogram function, is Rogers-Shephard inequality, which was originally proved in \cite[Thm. 1]{RS1} and states that for any convex body $K\subseteq\R^n$, we have
\begin{equation}\label{eq:RogersShephard}
|K-K| \leq \binom{2n}{n} |K|,
\end{equation}
with equality if and only if $K$ is a simplex.

Rogers-Shephard inequality was extended to the setting of log-concave functions in \cite[Thm. 2.2]{AGJV1}, where it was proved that for any $f\in\mathcal{F}(\R^n)$, calling $\tilde{f}(x):=f(-x)$, we have
\begin{equation}\label{eq:FunctionalRogersSephard}
\int_{\R^n}f\star\tilde{f}(x)dx\leq\binom{2n}{n}\Vert f\Vert_\infty\int_{\R^n}f(x)dx.
\end{equation}

In \cite[Thm. 1.1]{ALY}, the following discrete analogue of Rogers-Shephard inequality \eqref{eq:RogersShephard} was obtained for any non-empty convex bounded set $K$:
\begin{equation}\label{eq:Rogers-Shephard discrete}
G_n(K-K) \leq \binom{2n}{n} G_n \left( K+ \frac{3}{4}C_n \right).
\end{equation}

Following our purpose of getting discrete versions of functional inequalities, we will prove the following discrete version of the functional version of Rogers-Shephard inequality given by \eqref{eq:FunctionalRogersSephard}:
\begin{thm}\label{thm:DiscreteFunctionalRogersShephard}
Let $f\in\mathcal{F}(\R^n)$ and let $\tilde{f}(x)=f(-x)$. Then
$$
\int_{\R^n} f \star \tilde{f}(x) dG_n(x) \leq \binom{2n}{n} \norma{f}_{\infty} \int_{\R^n} f\star\chi_{\frac{3}{4}C_n}(x) dG_n(x).
$$
\end{thm}

The paper is organized as follows: In Section \ref{sec:Preliminaries} we will introduce some preliminary concepts which will be needed in order to prove Theorems \ref{thm:DiscreteFunctionalZhang} and \ref{thm:DiscreteFunctionalRogersShephard}. Since Berwald's inequality (see Section \ref{sec:Berwald} below) and its discrete version were the main tools to prove \eqref{eq:InclusionZhangRadialFunctions2} and \eqref{eq:InequalityDiscreteZhang}, Section \ref{sec:Discrete Berwald's inequalities} will be devoted to provide discrete versions of Berwald-type inequalities. On the one hand, it will be shown that the continuous and the discrete version of Berwald's inequality are equivalent, providing a new proof of the discrete version of Berwald's inequality (Theorem \ref{thm:BerwaldDiscreteNew}) which allows, in addition, to remove a condition on the maximum of the function in the discrete version of Berwald's inequality in \cite[Thm. 1.4]{ALY}. On the other hand, with a similar approach, we will provide a discrete version of Berwald's inequality in the setting of log-concave functions. In Section \ref{sec:FunctionalZhang} we will provide the proof of Theorem \ref{thm:DiscreteFunctionalZhang}. Finally, in Section \ref{sec:DiscreteFunctionalRogersShephard} we will prove Theorem \ref{thm:DiscreteFunctionalRogersShephard}.

\section{Preliminaries}\label{sec:Preliminaries}

\subsection{The lattice point enumerator}

The lattice point enumerator measure, $dG_n$, is the measure on $\R^n$ defined for any Borel set $A \subseteq \R^n$ as
\[
G_n(A)=\sharp(A\cap\Z^n),
\]
where $\sharp (\cdot)$ denotes the cardinality of a set. Whenever $A \subseteq \R^n$ is contained in the affine subspace $x_0+\text{span } \{e_1, \ldots, e_k \}$ for some $x_0 \in \text{span } \{e_1, \ldots, e_k \}^{\perp}$, we will denote $G_k(A)=G_n(A-x_0)$.

In \cite[Thm. 2.1]{IYNZ}, the authors proved that the lattice point enumerator satisfies the following discrete version of the Brunn-Minkowski inequality, from which one can recover the classical one \cite[Thm. 7.1.1]{Sch}:
\begin{thm}\label{thm: BM_lattice_point_no_G(K)G(L)>0}
Let $\lambda\in(0,1)$ and let $K,L\subset\R^n$ be non-empty bounded sets. Then
$$
G_n\left((1-\lambda)K+\lambda L+C_n\right)^{1/n}\geq(1-\lambda)G_n(K)^{1/n}+\lambda G_n(L)^{1/n}.
$$
\end{thm}

\begin{rmk}
Let us point out that Theorem \ref{thm: BM_lattice_point_no_G(K)G(L)>0} is true as long as $K,L\subset\R^n$ are non-empty bounded sets, even if $G_n(K)$ or $G_n(L)$ equals 0.
\end{rmk}

Notice that given a Borel set $A\subseteq\R^n$, for every $x\in A\cap\Z^n$ we have that $x+\frac{1}{2}B_\infty^n$ is a cube with volume $|x+\frac{1}{2}B_\infty^n|=1$ and if $x_1,x_2\in A\cap\Z^n$ with $x_1\neq x_2$, then  $(x_1+\frac{1}{2}B_\infty^n)\cap(x_2+\frac{1}{2}B_\infty^n)$ has volume 0. Taking also into account that $A\cap\Z^n\subseteq A$ and then $(A\cap\Z^n)+\frac{1}{2}B_\infty^n\subseteq A+\frac{1}{2}B_\infty^n$, we have that for any Borel set $A\subseteq\R^n$,
\begin{equation}\label{eq:LatticePointAndVolume}
G_n(A)=\left|(A\cap\Z^n)+\frac{1}{2}B_\infty^n\right|\leq\left|A+\frac{1}{2}B_\infty^n\right|.
\end{equation}

The lattice point enumerator measure also relates to Lebesgue measure in the following way (see \cite[Lem. 3.22]{TV} and \cite[Sect. 3.1]{ALY}): For any convex body $K \subseteq \R^n$ and any bounded set $M$ containing the origin
\begin{equation}\label{discrete_continuous_volume}
\lim_{\lambda \to \infty} \frac{G_n(\lambda K+M)}{\lambda^n} = |K|_n.
\end{equation}

In particular, taking $M=\{0\}$, for any convex body $K \subseteq \R^n$, we have
$$
\lim_{\lambda \to \infty} \frac{G_n(\lambda K)}{\lambda^n} = |K|_n.
$$

Moreover, for any $f: K \to \R$, a Riemann-integrable function on a convex body $K$,
\begin{equation} \label{Riemann_integrability}
\lim_{\lambda \to \infty} \frac{1}{\lambda^n} \int_{\lambda K} f\left( \frac{x}{\lambda} \right) dG_n(x) = \lim_{\lambda \to \infty} \frac{1}{\lambda^n} \sum_{x \in K \cap \left(\frac{1}{\lambda}\Z^n \right)} f(x) = \int_K f(x) dx.
\end{equation}

As a consequence of these equalities, many continuous inequalities can be recovered from the discrete versions obtained by using the discrete version of Brunn-Minkowski inequality given by Theorem \ref{thm: BM_lattice_point_no_G(K)G(L)>0}.

We will also consider along the paper the measure on $\R^{n}$ given by $d\mu(x)=d\mu_n(x):=dG_{n-1}\otimes dm_1(x)$, where $m_n$ denotes the Lebesgue measure on $\R^n$. Notice that for any $f:L\to[0,\infty)$ with $L\subseteq\R^n$, we have
$$
\int_{L}f(x)dG_n(x)=\mu_{n+1}(\textrm{hyp}(f)),
$$
where $\textrm{hyp}(f)$ is the hypograph of $f$ defined as
$$
\textrm{hyp}(f):=\{(x,t)\in L\times[0,\infty)\,:\, 0\leq t\leq f(x)\}.
$$
The measure $\mu$ satisfies (see \cite[Lem. 4.2]{ALM}) that for any bounded convex set $C\subseteq\R^n$
\begin{equation}\label{eq:MuAndGn}
G_n(C)-G_{n-1}(P_{e_n^\perp}(C))\leq \mu(C)\leq G_n(C)+G_{n-1}(P_{e_n^\perp}(C)).
\end{equation}
and, as a consequence, for any bounded convex set $C\subseteq\R^n$ and any bounded set $M$ containing the origin
$$
\lim_{\lambda \to\infty}\frac{\mu(\lambda C+M)}{\lambda^n}=|C|_n.
$$

\subsection{Log-concave functions}\label{sec:LogconcaveFunctions}

A function $f:\R^n\to[0,\infty)$ is called log-concave if for every $x_1,x_2\in\R^n$ and every $\lambda\in[0,1]$
$$
f((1-\lambda)x_1+\lambda x_2)\geq f(x_1)^{1-\lambda}f(x_2)^\lambda.
$$
Equivalently, $f$ is log-concave if $f(x)=e^{-u(x)}$ for some convex function $u:\R^n\to(-\infty,\infty]$. We will denote by $\mathcal{F}(\R^n)$ the set of integrable log-concave functions on $\R^n$ with positive integral. Equivalently, $\mathcal{F}(\R^n)$ is the set of integrable log-concave functions on $\R^n$ with full dimensional support. Let us point out that (see \cite[Lem. 2.2.1]{BGVV}), if $f\in\mathcal{F}(\R^n)$, then there exist $A,B>0$ such that $f(x)\leq A e^{-B\Vert x\Vert_2}$ for every $x\in\R^n$ and then $f$ has moments of all orders.

Given any function $u:\R^n\to(-\infty,\infty]$, its epigraph is the set
$$
\textrm{epi}(u):=\{(x,t)\in\R^{n+1}\,:\, u(x)\leq t\}=\{(x,t)\in D\times\R\,:\, u(x)\leq t\},
$$
where $D=\{x\in\R^n\,:\,u(x)\neq\infty\}$. Notice that  $u$ is a convex function if and only if $\textrm{epi}(u)$ is a convex set.



Along the paper, given any bounded log-concave function $f$ on $\R^n$, we will denote by $L(f)\subseteq \R^{n+1}$ the set
$$
L(f):=\{(x,t)\in\R^n\times[0,\infty)\,:\, f(x)\geq e^{-t}\Vert f\Vert_\infty\},
$$
which is the epigraph of the convex function $u:\R^n\to[0,\infty]$ such that $\frac{f(x)}{\Vert f\Vert_\infty}=e^{-u(x)}$ for every $x\in\R^n$. Therefore, $L(f)$ is a convex set for every $f\in\mathcal{F}(\R^n)$. We will also denote, for any $t\geq 0$,
$$
L_t(f):=\{x\in\R^n\,:\,\,f(x)\geq e^{-t}\Vert f\Vert_\infty\}=\{x\in\R^n\,:\,(x,t)\in L(f)\}.
$$
Notice that if $f\in\mathcal{F}(\R^n)$, then $L_t(f)$ is a bounded convex set for any $t\geq0$, that $L_{t_1}(f)\subseteq L_{t_2}(f)$ for every $0\leq t_1\leq t_2$, and that, $L_t(f)=\bigcap_{\varepsilon>0}L_{t+\varepsilon}(f)$ for any $t\geq0$.

We will also consider, along the paper, the measures on $\R^{n+1}$ defined as
\begin{itemize}
\item $d\nu(x,t)=d\nu_{n+1}(x,t):=dG_n(x)\otimes e^{-t} dt$,
\item $d\overline{\nu}(x,t)=d\overline{\nu}_{n+1}(x,t):=dm_n(x)\otimes e^{-t} dt$,
\item $d\sigma(x,t)=d\sigma_{n+1}(x,t):=d\mu_n(x)\otimes e^{-t}dt$.
\end{itemize}
If $f\in\mathcal{F}(\R^n)$, the measure $\nu$ of $L(f)$ equals
\begin{equation}\label{eq:nu(L)Integral}
\nu(L(f))=\int_0^\infty e^{-t} G_n(L_t(f))dt=\int_0^1\int_{\R^n}\chi_{\left\{x\in\R^n\,:\,\frac{f(x)}{\Vert f\Vert_\infty} \geq s \right\}}(x)dG_n(x)ds=\int_{\R^n}\frac{f(x)}{\Vert f\Vert_\infty}dG_n(x),
\end{equation}
the measure $\overline{\nu}$ of $L(f)$ equals
\begin{equation}\label{eq:barnu(L)Integral}
\overline{\nu}(L(f))=\int_0^\infty e^{-t} |L_t(f)|dt=\int_0^1\int_{\R^n}\chi_{\left\{x\in\R^n\,:\,\frac{f(x)}{\Vert f\Vert_\infty}\geq s\right\}}(x)dxds=\int_{\R^n}\frac{f(x)}{\Vert f\Vert_\infty}dx,
\end{equation}
and the measure $\sigma$ of $L(f)$ equals
\begin{equation}\label{eq:sigma(L)Integral}
\sigma(L(f))=\int_0^\infty e^{-t} \mu_n(L_t(f))dt=\int_0^1\int_{\R^n}\chi_{\left\{x\in\R^n\,:\,\frac{f(x)}{\Vert f\Vert_\infty}\geq s\right\}}(x)d\mu_n(x)ds=\int_{\R^n}\frac{f(x)}{\Vert f\Vert_\infty}d\mu_n(x).
\end{equation}

\subsubsection{The Asplund product of log-concave functions}\label{sec:AsplundProduct}

Given $f$ and $g$ log-concave functions, the Asplund product of $f$ and $g$ is the function
$$
f\star g(x):=\sup_{z\in\R^n}f(z)g(x-z)=\sup_{z_1+z_2=x\atop z_1,z_2\in\R^n}f(z_1)g(z_2).
$$
Notice that $\Vert f\star g\Vert_\infty=\Vert f\Vert_\infty\Vert g\Vert_\infty$ an that, as a consequence of \cite[Thm. 2.1]{AGJV1}, if $f,g\in\mathcal{F}(\R^n)$, then $f\star g\in\mathcal{F}(\R^n)$. Besides, we have that
$$
L(f)+L(g)\subseteq L(f\star g)
$$
and, for every $t_1, t_2\geq0$,
$$
L_{t_1}(f)+L_{t_2}(g)\subseteq L_{t_1+t_2}(f\star g).
$$
Thus, if $g=\chi_C$ for some convex set $C\subseteq\R^n$, we have that $L_{\varepsilon}(g)=C$ for every $\varepsilon\geq0$. Therefore, for every $t\geq0$
$$
L_{t}(f)+C\subseteq L_{t+\varepsilon}(f\star \chi_C)
$$
and, consequently,
\begin{equation}\label{eq:InclusionSuperlevelSetsAsplund}
L_{t}(f)+C\subseteq \bigcap_{\varepsilon>0}L_{t+\varepsilon}(f\star \chi_C)=L_t(f\star\chi_C).
\end{equation}

On the other hand, for any $t\geq0$ and any $x\in L_t(f\star\chi_C)$, there exists a sequence $(z_n)_{n=1}^\infty\in C$ such that
$$
\lim_{n\to\infty}f(x-z_n)=(f\star\chi_C)(x)\geq e^{-t}\Vert f\Vert_\infty.
$$
Moreover, if $C$ is bounded we can assume that $(z_n)_{n=1}^\infty$ converges to some $z\in\overline{C}$, the closure of $C$. Therefore, for every $\varepsilon>0$ there exists $n_0\in\N$ such that if $n\geq n_0$ we have that $x-z_n\in L_{t+\varepsilon}(f)$ and then $x-z\in \overline{L_{t+\varepsilon}(f)}$ for every $\varepsilon>0$. Consequently, if $C$ is bounded we have that
\begin{equation}\label{eq:InverseInclusionSuperLevelSetsAsplund}
L_t(f\star\chi_C)\subseteq\bigcap_{\varepsilon>0}\overline{L_{t+\varepsilon}(f)}+\overline{C}.
\end{equation}

As a consequence of \eqref{eq:InclusionSuperlevelSetsAsplund}, we have the following lemma which ensures that if $f\in\mathcal{F}(\R^n)$, then $f$ is integrable respect to the lattice point enumerator measure, $dG_n$.

\begin{lemma}\label{lem:Integrability}
Let $f\in\mathcal{F}(\R^n)$. Then $f$ is integrable respect to $dG_n$ and
$$
\int_{\R^n}f(x)dG_n(x)\leq \int_{\R^n}f\star\chi_{\frac{1}{2}B_\infty^n}(x)dx.
$$
\end{lemma}

\begin{proof}
Let $f\in\mathcal{F}(\R^n)$. Then by \eqref{eq:nu(L)Integral}, \eqref{eq:LatticePointAndVolume}, \eqref{eq:InclusionSuperlevelSetsAsplund} and \eqref{eq:barnu(L)Integral}, we have
\begin{eqnarray*}
\int_{\R^n}\frac{f(x)}{\Vert f\Vert_\infty}dG_n(x)&=&\nu(L(f))=\int_0^\infty e^{-t} G_n(L_t(f))dt\leq\int_0^\infty e^{-t}\left|L_t(f)+\frac{1}{2}B_\infty^n\right|dt\cr
&\leq&\int_0^\infty e^{-t}\left|L_t(f\star\chi_{\frac{1}{2}B_\infty^n})\right|dt=\overline{\nu}\left(L\left(f\star\chi_{\frac{1}{2}B_\infty^n}\right)\right)\cr
&=&\int_{\R^n}\frac{f\star\chi_{\frac{1}{2}B_\infty^n}(x)}{\Vert f\Vert_\infty}dx.
\end{eqnarray*}
Since $f\star\chi_{\frac{1}{2}B_\infty^n} \in\mathcal{F}(\R^n)$, the last integral is finite.
\end{proof}

Given $f\in\mathcal{F}(\R^n)$ and $\lambda>0$, we will denote by $f_\lambda:\R^n\to[0,\infty)$ the function given by $f_\lambda(x)=f\left(\frac{x}{\lambda}\right)$. Notice that $f_{\lambda} \in \mathcal{F}(\R^n)$, $\Vert f_\lambda\Vert_\infty=\Vert f\Vert_\infty$ and that for any $t\geq 0$ we have
$$
L_t(f_\lambda)=\{ x \in \R^n : f_\lambda(x) \geq e^{-t} \norma{f_\lambda}_{\infty} \}=\left\{x\in\R^n\,:\,\frac{x}{\lambda}\in L_t(f)\right\}=\lambda L_t(f).
$$
Notice also that for any $f, g\in\mathcal{F}(\R^n)$ and any $\lambda>0$ we have
\begin{equation}\label{eq:IdentityLambda}
f_{\lambda} \star g_{\lambda} (x)= \sup_{z\in\R^n}f_{\lambda}(z) g_{\lambda}(x-z) = \sup_{z\in\R^n}f\left(\frac{z}{\lambda} \right) g\left(\frac{x-z}{\lambda} \right) = f \star g \left( \frac{x}{\lambda} \right)=(f\star g)_\lambda(x).
\end{equation}

The following lemma shows that \eqref{Riemann_integrability} holds also when $f\in\mathcal{F}(\R^n)$ does not necessarily have compact support.
\begin{lemma}\label{lem:LogConcaveIntegralDiscreteToContinuous}
Let $f\in\mathcal{F}(\R^n)$ be continuous on its support and, for any $\lambda>0$, let  $f_\lambda\in\mathcal{F}(\R^n)$ be the function defined as $f_\lambda(x)= f\left(\frac{x}{\lambda}\right)$. Let also $M$ be a bounded set containing the origin. Then,
$$
\lim_{\lambda\to\infty}\frac{1}{\lambda^n}\int_{\R^n}f_\lambda\star\chi_{M}(x)dG_n(x)=\int_{\R^n}f(x)dx.
$$
In particular, taking $M=\{0\}$,
$$
\lim_{\lambda\to\infty}\frac{1}{\lambda^n}\int_{\R^n}f_\lambda(x)dG_n(x)=\int_{\R^n}f(x)dx.
$$
\end{lemma}

\begin{proof}
Notice that, since $\Vert f_\lambda\star\chi_M\Vert_\infty=\Vert f_\lambda\Vert_\infty=\Vert f\Vert_\infty$, by \eqref{eq:nu(L)Integral},
$$
\lim_{\lambda \to \infty} \int_{\R^n} \frac{f_\lambda\star\chi_M (x)}{\Vert f\Vert_\infty} dG_n(x)=\lim_{\lambda\to\infty}\int_0^\infty e^{-t} \frac{G_n (L_t (f_\lambda\star\chi_M))}{\lambda^n}dt.
$$

On the one hand, since $f$ is continuous on its support, we have that $L(f)$ is closed. Thus, by \eqref{eq:InclusionSuperlevelSetsAsplund} and \eqref{eq:InverseInclusionSuperLevelSetsAsplund}, we have that for every $t\geq0$ and every $\lambda>0$
$$
L_t(f_\lambda)+M\subseteq L_t(f_\lambda\star\chi_M)\subseteq\bigcap_{\varepsilon>0}\overline{L_{t+\varepsilon}(f_\lambda)}+\overline{M}=\bigcap_{\varepsilon>0}L_{t+\varepsilon}(f_\lambda)+\overline{M}=L_t(f_\lambda)+\overline{M}.
$$
Therefore, for every $t\geq0$ and every $\lambda>0$,
$$
\lambda L_t(f)+M\subseteq L_t(f_\lambda\star\chi_M)\subseteq \lambda L_t(f)+\overline{M}.
$$
Since by \eqref{discrete_continuous_volume}, for every $t\geq0$ we have that
$$
\lim_{\lambda\to\infty}e^{-t} \frac{G_n (\lambda L_t (f)+M)}{\lambda^n}=e^{-t} |L_t (f)|\quad\textrm{and}\quad\lim_{\lambda\to\infty}e^{-t} \frac{G_n (\lambda L_t (f)+\overline{M})}{\lambda^n}=e^{-t} |L_t (f)|,
$$
we obtain that for every $t\geq0$
$$
\lim_{\lambda\to\infty}e^{-t}\frac{G_n (L_t(f_\lambda\star\chi_M))}{\lambda^n}=e^{-t}|L_t(f)|.
$$

On the other hand, since $M$ is bounded, there exists $R>0$ such that $\overline{M}\subseteq RB_\infty^n$. Therefore,  using \eqref{eq:LatticePointAndVolume} and \eqref{eq:InclusionSuperlevelSetsAsplund}, for every $t>0$ and every $\lambda\geq R+\frac{1}{2}$
\begin{eqnarray*}
G_n(L_t(f_\lambda\star\chi_M))&\leq& G_n(\lambda L_t(f)+\overline{M})\leq\left|\lambda L_t(f)+\overline{M}+\frac{1}{2}B_\infty^n\right|\cr
&\leq&\left|\lambda L_t(f)+\left(R+\frac{1}{2}\right)B_\infty^n\right|\leq\left|\lambda L_t(f)+\lambda B_\infty^n\right|=\lambda^n|L_t(f)+B_\infty^n|\cr
&\leq&\lambda^n L_t(f\star\chi_{B_\infty^n}).
\end{eqnarray*}
 Thus, for every $t\geq0$ and every $\lambda>R+\frac{1}{2}$ we have
$$
e^{-t} \frac{G_n(L_t(f_\lambda\star\chi_M))}{\lambda^n}\leq e^{-t}\left|L_t (f\star \chi_{B_\infty^n})\right|,
$$
which, taking into account \eqref{eq:barnu(L)Integral}, is integrable on $[0,\infty)$, since $f\star \chi_{B_\infty^n}\in\mathcal{F}(\R^n)$.

By the dominated convergence theorem,
$$
\lim_{\lambda\to\infty}\int_0^\infty e^{-t} \frac{G_n(L_t(f_\lambda\star\chi_M))}{\lambda^n}dt=\int_0^\infty e^{-t} |L_t (f)|dt=\int_{\R^n}\frac{f(x)}{\Vert f\Vert_\infty}dx,
$$
which proves the lemma.
\end{proof}

\subsubsection{Projections of log-concave functions}\label{sec:ProjectionsOfFunctions}

Given $f\in\mathcal{F}(\R^n)$ and a $k$-dimensional linear subspace in $\R^n$, $H\in G_{n,k}$, where $G_{n,k}$ denotes the set of $k$-dimensional linear subspaces in $\R^n$, the projection of $f$ onto $H$ was defined in \cite{KM}, and it is the log-concave function $P_Hf:H\to[0,\infty)$ given by
$$
P_Hf(y)=\sup_{z\in H^\perp}f(y+z).
$$
Notice  that $\Vert P_Hf\Vert_\infty=\Vert f\Vert_\infty$. Besides, as a direct consequence of \cite[Lem. 2.1.1]{BGVV}, $P_Hf\in\mathcal{F}(H)$, the set of integrable log-concave functions in $H$ with positive integral. Besides, the set $L(P_Hf)=\{(y,t)\in H\times[0,\infty)\,:\,P_Hf(y)\geq e^{-t}\Vert f\Vert_\infty \}$ satisfies that $L(P_Hf)=P_H(L(f))$ and, therefore, for any $t\geq0$, $L_t(P_Hf):=\{y\in H\,:\,P_H f\geq e^{-t}\Vert f\Vert_\infty\}=P_H(L_t(f))$. Let us point out that, if $f=\chi_K$ for some convex set $K\subseteq\R^n$, then $P_Hf=\chi_{P_H(K)}$ for any linear subspace $H$ and if $f=e^{-\Vert \cdot\Vert_K}$ for some convex body $K$ with $0\in K$, then $P_Hf=e^{-\Vert \cdot\Vert_{P_H(K)}}$ for any linear subspace $H$.

Notice also that, by \eqref{eq:MuAndGn}, if $f\in\mathcal{F}(\R^n)$ then
\begin{eqnarray*}
\sigma(L(f))&=&\int_0^\infty e^{-t} \mu_n(L_t(f))dt\leq \int_0^\infty e^{-t}(G_n(L_t(f))+G_{n-1}(P_{e_n^\perp}(L_t(f)))dt\cr
&=&\int_0^\infty e^{-t}(G_n(L_t(f))+G_{n-1}(L_t(P_{e_n^\perp}f)))dt=\nu_n(L(f))+\nu_{n-1}(L(P_{e_n^\perp}f))
\end{eqnarray*}
and, similarly
\begin{eqnarray*}
\sigma(L(f))&=&\int_0^\infty e^{-t} \mu_n(L_t(f))dt\geq \int_0^\infty e^{-t}(G_n(L_t(f))-G_{n-1}(P_{e_n^\perp}(L_t(f)))dt\cr
&=&\int_0^\infty e^{-t}(G_n(L_t(f))-G_{n-1}(L_t(P_{e_n^\perp}f)))dt=\nu_n(L(f))-\nu_{n-1}(L(P_{e_n^\perp}f)).
\end{eqnarray*}
Therefore, if $f\in\mathcal{F}(\R^n)$, we have that $P_{e_n^\perp}f\in\mathcal{F}(e_n^\perp)$ and, by Lemma \ref{lem:Integrability}, both $f$ and $P_{e_n^\perp}f$ are integrable on $\R^n$ and $e_n^\perp$ with respect to $dG_n$ and $dG_{n-1}$ and, therefore, $f$ is integrable on $\R^n$ with respect to $d\mu$ with
\begin{eqnarray}\label{eq:integralMuandm}
\int_{\R^n}\frac{f(x)}{\Vert f\Vert_\infty}dG_n(x)-\int_{e_n^\perp}\frac{P_{e_n^\perp}f(x)}{\Vert f\Vert_\infty}dG_{n-1}(x)&\leq&\int_{\R^n}\frac{f(x)}{\Vert f\Vert_\infty}d\mu(x)\cr
&\leq&\int_{\R^n}\frac{f(x)}{\Vert f\Vert_\infty}dG_n(x)+\int_{e_n^\perp}\frac{P_{e_n^\perp}f(x)}{\Vert f\Vert_\infty}dG_{n-1}(x).
\end{eqnarray}

As a consequence, we also have the following lemma, analogous to Lemma \ref{lem:LogConcaveIntegralDiscreteToContinuous},
\begin{lemma}\label{lem:LogConcaveIntegralMuToContinuous}
Let $f\in\mathcal{F}(\R^n)$ be continuous on its support and, for any $\lambda>0$, let $f_\lambda\in\mathcal{F}(\R^n)$ be the function defined as $f_\lambda(x)= f\left(\frac{x}{\lambda}\right)$. Let also $M$ be a bounded set containing the origin. Then,
$$
\lim_{\lambda\to\infty}\frac{1}{\lambda^n}\int_{\R^n}f_\lambda\star\chi_{M}(x)d\mu(x)=\int_{\R^n}f(x)dx.
$$
In particular, taking $M=\{0\}$,
$$
\lim_{\lambda\to\infty}\frac{1}{\lambda^n}\int_{\R^n}f_\lambda(x)d\mu(x)=\int_{\R^n}f(x)dx.
$$
\end{lemma}

\begin{proof}
By \eqref{eq:integralMuandm} and Lemma \ref{lem:LogConcaveIntegralDiscreteToContinuous}, it is enough to see that
$$
\lim_{n\to\infty}\frac{1}{\lambda^n}\int_{e_n^\perp}P_{e_n^\perp}(f_\lambda\star\chi_{M})(y)dG_{n-1}(y)=0.
$$
Notice that, like in the proof of Lemma \ref{lem:LogConcaveIntegralDiscreteToContinuous}, we have by \eqref{eq:InverseInclusionSuperLevelSetsAsplund} that for every $t\geq0$ and every $\lambda>0$
$$
L_t(f_\lambda\star\chi_M)\subseteq \lambda L_t(f)+\overline{M}.
$$
Therefore, since for every $\lambda>0$ we have $P_{e_n^\perp}f_\lambda=(P_{e_n^\perp}f)_\lambda$,  taking into account \eqref{eq:InclusionSuperlevelSetsAsplund}
\begin{eqnarray*}
L_t(P_{e_n^\perp}(f_\lambda\star\chi_M))&= &P_{e_n^\perp}L_t(f_\lambda\star\chi_M)\subseteq \lambda P_{e_n^\perp}L_t(f)+P_{e_n^\perp}\overline{M}= \lambda L_t(P_{e_n^\perp}f)+P_{e_n^\perp}\overline{M}\cr
&=&L_t((P_{e_n^\perp}f)_\lambda)+P_{e_n^\perp}\overline{M}\subseteq L_t((P_{e_n^\perp}f)_\lambda\star\chi_{P_{e_n^\perp}\overline{M}}).
\end{eqnarray*}
Therefore, since  $\Vert P_{e_n^\perp}(f_\lambda\star\chi_M)\Vert_\infty=\Vert f_\lambda\star\chi_M\Vert_\infty=\Vert f_\lambda\Vert_\infty=\Vert f\Vert_\infty$ and
$\Vert(P_{e_n^\perp}f)_\lambda\star\chi_{P_{e_n^\perp}\overline{M}}\Vert_\infty=\Vert (P_{e_n^\perp}f)_\lambda\Vert_\infty=\Vert P_{e_n^\perp}f\Vert_\infty=\Vert f\Vert_\infty$ by \eqref{eq:nu(L)Integral} we obtain that
$$
0\leq\int_{e_n^\perp}P_{e_n^\perp}(f_\lambda\star\chi_M)(y)dG_{n-1}(y)\leq\int_{e_n^\perp}(P_{e_n^\perp}f)_\lambda\star\chi_{P_{e_n^\perp}\overline{M}}(y)dG_{n-1}(y).
$$
Since $P_{e_n^\perp}f\in\mathcal{F}(e_n^\perp)$ and $P_{e_n^\perp}\overline{M}$ is a bounded set containing the origin, by Lemma \ref{lem:LogConcaveIntegralDiscreteToContinuous} we have, identifying $e_n^\perp$ with $\R^{n-1}$, that
$$
\lim_{\lambda\to\infty}\frac{1}{\lambda^{n-1}}\int_{e_n^\perp}(P_{e_n^\perp}f)_\lambda\star\chi_{P_{e_n^\perp}\overline{M}}(y)dG_{n-1}(y)=\int_{e_n^\perp}P_{e_n^\perp}f(y)dy,
$$
which implies that
$$
\lim_{\lambda\to\infty}\frac{1}{\lambda^n}\int_{e_n^\perp}(P_{e_n^\perp}f)_\lambda\star\chi_{P_{e_n^\perp}\overline{M}}(y)dG_{n-1}(y)=0.
$$

\end{proof}
\subsubsection{Non-negative concave functions}

Along the paper, we will also consider non-negative concave functions $h:L\to[0,\infty)$, where $L\subseteq\R^n$ is a convex set. That is, $h$ satisfies that, for any $x_1,x_2\in L$ and any $\lambda\in[0,1]$,
$$
h((1-\lambda)x_1+\lambda x_2)\geq (1-\lambda)h(x_1)+\lambda h(x_2).
$$
Through the paper, $L$ will usually be either a bounded convex set or a set of the form $L(f)$ for some $f\in\mathcal{F}(\R^n)$. Notice that $h$ is concave if and only if $\textrm{hyp}(f)$ is convex.

Defining the function $\overline{h}:\R^n\to[0,\infty)$ given by
$$
\overline{h}(x)=\begin{cases}
h(x) &\textrm{ if }x\in L\\
0 &\textrm{ if }x\not\in L,
\end{cases}
$$
we have that, if $h$ is concave, as a consequence of the arithmetic-geometric mean inequality, for every $x_1,x_2\in\R^n$
$$
\overline{h}((1-\lambda)x_1+\lambda x_2)\geq \overline{h}(x_1)^{1-\lambda}\overline{h}(x_2)^\lambda
$$
and then, $\overline{h}$ is log-concave. Along the paper, we will make the abuse of notation of also denoting by $h$ the function $\overline{h}$. In particular, given a concave function $h:L\to[0,\infty)$ with $L\subseteq\R^n$ a convex set and a convex set $C\subseteq\R^n$, we will consider the function $h^\diamond:L+C\to[0,\infty)$ given by

\begin{equation}\label{eq:DefinitionfDiamond}
h^\diamond(x)=h\star\chi_C(x)=\sup_{u\in C}\overline{h}(x-u)=\sup_{u\in C\atop x-u\in L}h(x-u).
\end{equation}

The function $h^\diamond$ is concave on $L+C$. Given any concave function $h:L\to[0,\infty)$ defined on a convex set $L\subseteq\R^n$, we will denote, for any $s\geq 0$, the set
\begin{equation}\label{eq:DefinitionC_s}
C_s(h):=\{x\in L\,:\, h(x)\geq s\}.
\end{equation}
Notice that if $h$ is bounded, then $C_s(h)=L_{-\log\left(\frac{s}{\Vert h\Vert_\infty}\right)}(h)$ for every $s\in[0,\Vert h\Vert_\infty]$. Notice also that for any convex set $C\subseteq\R^n$, if $h^\diamond$ is defined as in \eqref{eq:DefinitionfDiamond}, we have that $\textrm{hyp}(h^\diamond)\supseteq \textrm{hyp}(h)+C$ and then, for every $s\geq 0$,
\begin{equation}\label{eq:InclusionSuperlevelSetsAsplundConcave}
C_s(h^\diamond)\supseteq C_s(h)+C.
\end{equation}
On the other hand, if $s>0$, notice that for any $x\in C_s(h^\diamond)$ there exists a sequence $(u_n)_{n=1}^\infty\subseteq C$ such that $x-u_n\in L$ and
$$
\lim_{n\to\infty} h(x-u_n)=h^\diamond(x)\geq s.
$$
Moreover, if $C$ is bounded we can assume that $(u_n)_{n=1}^\infty$ converges to some $u\in\overline{C}$. Therefore, for any $\varepsilon>0$ we have that there exists $n_0\in\N$ such that if $n\geq n_0$ then $x-u_n\in C_{s-\varepsilon}(h)$ and, consequently, $x-u\in \overline{C_{s-\varepsilon}(h)}$ for every $\varepsilon>0$. Therefore, if $C$ is bounded, we have that
\begin{equation}\label{eq:InverseInclusionSuperlevelSetsAsplundConcave}
C_s(h^\diamond)\subseteq \bigcap_{\varepsilon>0}\overline{C_{s-\varepsilon}(h)}+\overline{C}.
\end{equation}

\subsection{The covariogram function}

Given a convex body $K\subseteq\R^n$, its covariogram function is the function $g_K:\R^n\to[0,\infty)$, supported on $K-K$, which is defined as
$$
g_K(x):=|K\cap(x+K)|=\int_{\R^n}\chi_K(z)\chi_{K}(z-x)dz=\int_{\R^n}\min\left\{\chi_K(z),\chi_{K}(z-x)\right\}dz.
$$

The function $g_K$ satisfies that $g_K^\frac{1}{n}$ is concave on $K-K$, as a consequence of Brunn-Minkowski inequality, and therefore $g_K$ is log-concave on $\R^n$. Besides, $g_K$ is an even function that satisfies $\Vert g_K\Vert_\infty=g(0)=|K|$ and
\begin{equation}\label{eq:IntegralCovariogram}
\int_{\R^n} g_K(x)=\int_{\R^n}\int_{\R^n}\min\left\{\chi_K(x),\chi_{K}(y)\right\}dxdy=|K|^2.
\end{equation}

Given $f\in\mathcal{F}(\R^n)$, its covariogram function $g_f:\R^n\to[0,\infty)$ is the function defined as
$$
g_f(x):=\int_{\R^n}\min\left\{\frac{f(z)}{\Vert f\Vert_\infty},\frac{f(z-x)}{\Vert f\Vert_\infty}\right\}dz.
$$
This function satisfies (see \cite[Lem. 2.2.1]{ABG1}) that $g_f\in\mathcal{F}(\R^n)$ and
\begin{eqnarray}\label{eq:Covariogram}
g_f(x)&=&\int_0^\infty e^{-t}g_{L_t(f)}(x)dt=\overline{\nu}(L(f)\cap(x+ L(f)).
\end{eqnarray}
Moreover, $g_f$ is an even function,
$$
\Vert g_f\Vert_\infty=g_f(0)=\overline{\nu}(L(f))=\int_{\R^n}\frac{f(y)}{\Vert f\Vert_\infty}dy
$$
and
\begin{equation}\label{eq:IntegralFunctionalCovariogram}
\int_{\R^n}g_f(x)dx=\int_{\R^n}\int_{\R^n}\min\left\{\frac{f(x)}{\Vert f\Vert_\infty},\frac{f(y)}{\Vert f\Vert_\infty}\right\}dydx.
\end{equation}
Let us point out that if $f=\chi_K$ for some convex body $K\subseteq\R^n$, then $g_f=g_K$.

\subsection{Ball's bodies of log-concave functions}

In \cite{Ba}, Ball introduced, for any measurable function $g:\R^n\to[0,\infty)$, such that $g(0)>0$, and for any $p>0$, the set
$$
K_p(g):=\left\{x\in\R^n\,:\,p\int_0^\infty r^{p-1}g(rx)dr\geq g(0)\right\}.
$$
These sets satisfy that, for any $p>0$, they are star sets with center $0$ and radial function given, for any $\theta\in S^{n-1}$
$$
\rho_{K_p(g)}(\theta)=\left(\frac{p}{g(0)}\int_0^\infty r^{p-1}g(r\theta)dr\right)^\frac{1}{p}.
$$
The importance of these sets $(K_p(g))_{p>0}$, which we will call $p$-th Ball bodies of $g$, in the study of log-concave functions relies on the following two facts: First, whenever $g:\R^n\to[0,\infty)$ is an integrable log-concave function such that $g(0)>0$, the star set $K_p(g)$ is a convex body for any $p>0$ (see \cite[Thm. 2.5.5, Lem. 2.5.6, and Prop. 2.5.7]{BGVV}). As a particular case, notice that if $K\subseteq\R^n$ is a convex body with $0\in K$, then for any $p>0$
$$
K_p(\chi_K)=K\quad\textrm{and}\quad K_p\left(e^{-\Vert\cdot\Vert_K}\right)=\frac{K}{\Gamma\left(1+p\right)^{1/p}}.
$$

Second, for any homogeneous function $h:\R^n\to[0,\infty)$ of degree 1 and any $p>-n$ we have, by integration in polar coordinates (see \cite[Prop. 2.5.3]{BGVV} for the particular case when $h$ is a norm on $\R^n$), that
$$
\int_{K_{n+p}(g)}h^p(x)dx=\int_{\R^n}h^p(x)\frac{g(x)}{g(0)}dx.
$$

In particular, if $g:\R^n\to[0,\infty)$ is an integrable log-concave function, such that $g(0)>0$, taking $p=0$ (see \cite[Lem. 2.5.6]{BGVV}) we obtain that  $\displaystyle{|K_n(g)|=\int_{\R^n}\frac{g(x)}{g(0)}dx}$.

If $K\subseteq\R^n$ is convex body, taking $g=g_K$ its covariogram function, then the $n$-th Ball body of $K$, $K_n(g_K)$, is a convex body whose volume is, by \eqref{eq:IntegralCovariogram},
$$
|K_n(g_K)|=\int_{\R^n}\frac{g_K(x)}{g_K(0)}dx=|K|
$$
and, taking volumes in \eqref{eq:InclusionZhang}, Zhang's inequality, \eqref{eq:Zhang}, is obtained.

If $f\in\mathcal{F}(\R^n)$, taking $g=g_f$ its covariogram function, then the $n$-th Ball body of $K$, $K_n(g_K)$, is a convex body whose volume is, by \eqref{eq:IntegralFunctionalCovariogram}
$$
|K_n(g_f)|=\int_{\R^n}\frac{g_f(x)}{g_K(0)}dx=\frac{1}{\Vert f\Vert_1}\int_{\R^n}\int_{\R^n}\min\left\{f(x),f(y)\right\}dydx
$$
and, taking volumes in \eqref{eq:InclusionFunctionalZhang}, the functional version of Zhang's inequality, \eqref{eq:FunctionalZhang} is obtained.

For any integrable log-concave function $g:\R^n\to[0,\infty)$ such that $g(0)>0$ the following inclusion relation between Ball bodies holds (see \cite[Prop. 2.5.7]{BGVV}): If $0<p<q$, then
$$
\frac{1}{\Gamma (1+q)^\frac{1}{q}} K_q (g) \subseteq  \frac{1}{\Gamma (1+p)^\frac{1}{p}}K_p (g),
$$
which follows from the fact that (see \cite[Thm.2.2.3]{BGVV}), for any log-concave function $g:[0,\infty)\to[0,\infty)$ with $g(0)>0$  and any $0<p<q$, we have
\begin{equation}\label{eq:InequalityInclusionBallsBodies}
\left( \frac{1}{\Gamma(1+q) g(0)}  \int_0^{\infty}qs^{q-1}g(s) ds \right)^{1/q} \leq \left( \frac{1}{\Gamma(1+p) g(0) }  \int_0^{\infty} ps^{p-1}g(s) ds \right)^{1/p}.
\end{equation}

We refer the reader to \cite[Sect. 2.5]{BGVV} for more information on the family of $p$-th Ball bodies.

\subsection{Steiner's symmetrization}\label{sec:Steiner}

Given a bounded convex set $K\subseteq\R^n$, the Steiner symmetrization of $K$ with respect to the hyperplane $e_n^\perp$ is defined as
$$
S_{e_n}(K)=\left\{y+\frac{t_1-t_2}{2}e_n\,:\,y\in P_{e_n^\perp}(K),\,y+t_1e_n\in K,\,y+t_2e_n\in K\right\}.
$$
That is, $S_{e_n}(K)$ is the set that we obtain by first, shifting all the segments given by $K\cap(y+\langle e_n\rangle)$ in the direction parallel to $\langle e_n\rangle$ until their centers lie in the hyperplane $e_n^\perp$, and second, leaving such segments closed if they were closed and open otherwise. If $K$ is compact then $S_{e_n}(K)$ can be written as
$$
S_{e_n}(K)=\left\{(y,t)\in\R^{n-1}\times\R\,:\,y\in P_{e_n^\perp}(K),\,|t|\leq\frac{|K\cap(y+\langle e_n\rangle)|_1}{2}\right\}.
$$
The Steiner symmetrization preserves convexity and volume. Moreover, for every $y\in P_{e_n^\perp}(K)$ we have that $S_{e_n}(K)\cap(y+\langle e_n\rangle)$ is an interval centered at $y$ which has the same length as $K\cap(y+\langle e_n\rangle)$. Besides, from the definition of $S_{e_n}(K)$, if $K\subseteq\R^n$ is a convex body then
$$
S_{e_n}(K)\cap\{x\in\R^n\,:\,\langle x,e_n\rangle\geq0\}
$$
is the hypograph of the function $f:P_{e_n^\perp}(K)\to[0,\infty)$ given by $$f(y)=\frac{\vol_1(K\cap(y+\langle e_n\rangle))}{2},$$ which is concave by Brunn's principle (see, for instance \cite[Thm. 1.2.2]{BGVV}). It is also known that for any bounded convex set $K$ and any $\lambda\geq0$ we have that  $S_{e_n}(\lambda K)=\lambda S_{e_n}(K)$ and that for any two bounded convex sets $K,L\subseteq\R^n$ one has that
\begin{equation}\label{eq:SteinerSymmetrizationInclusion}
S_{e_n}(K)+S_{e_n}(L)\subseteq S_{e_n}(K+L).
\end{equation}
A list of basic properties of the Steiner symmetrization of convex bodies can be found in \cite[Sect. 1.1.7 and A.5]{AGM}.

Given a log-concave integrable function $f\in\mathcal{F}(\R^n)$, the Steiner symmetrization of $L(f)\subseteq\R^n\times[0,\infty)$ with respect $e_n^\perp$ is defined as
$$
S_{e_n}(L(f))=\{(x,t)\in\R^n\times[0,\infty)\,:\, x\in S_{e_n}(L_t(f))\}.
$$

Notice that since $L(f)$ is convex, $S_{e_n}(L(f))$ is a convex set and that if $(x,t_0)\in S_{e_n}(L(f))$, then $(x,t)\in S_{e_n}(L(f))$ for every $t\geq t_0$. The Steiner symmetrization of $f$ with respect $e_n^\perp$ is then defined as
$$
S_{e_n}(f)(x)=\Vert f\Vert_\infty e^{-\tilde{u}(x)},
$$
where $\tilde{u}:\R^n\to[0,\infty]$ is the convex function given by
$$
\tilde{u}(x)=\inf\{t\geq0\,:\, (x,t)\in S_{e_n}(L(f))\},
$$
understanding the infimum as $\infty$ if $\{t\geq0\,:\, (x,t)\in S_{e_n}(L(f))\}=\emptyset$. Notice that $\Vert S_{e_n}(f)\Vert_\infty=\Vert{f}\Vert_{\infty}$ and that, even though $S_{e_n}(L(f))\subseteq L(S_{e_n}(f))$ and the inclusion can be strict, for every $(y,t)\in e_n^\perp\times[0,\infty)\subseteq\R^{n+1}$ we have that
$$
|L(S_{e_n}(f))\cap ((y,t)+\langle e_n\rangle)|_1=|S_{e_n}(L(f))\cap ((y,t)+\langle e_n\rangle)|_1=|L_t(f)\cap (y+\langle e_n\rangle)|_1.
$$

\subsection{Berwald's inequality}\label{sec:Berwald}

Berwald's inequality provides a reverse H\"older's inequality for $L^p$-norms of concave functions on convex bodies. It is stated in the following Theorem:

\begin{thm}[Berwald's inequality]\label{thm:Berwald}
Let $K\subseteq\R^n$ be a convex body and let $h:K\to[0,\infty)$ be a concave function. Then, for any $-1<p\leq q$ we have that
$$
\left(\frac{{{n+q}\choose{n}}}{|K|}\int_K h^q(x)dx\right)^{1/q}\leq\left(\frac{{{n+p}\choose{n}}}{|K|}\int_K h^p(x)dx\right)^{1/p}.
$$
\end{thm}

Berwald's inequality was proved in \cite[Satz 7]{Be} whenever the parameters in the statement satisfy $0<p\leq q$ (see also \cite[Thm. 7.2]{AAGJV} for an English translation). It was extended to the whole range $-1<p\leq q$ in \cite[Thm. 5.1]{GZ}.

With the use of the discrete version of the Brunn-Minkowski inequality, Theorem \ref{thm: BM_lattice_point_no_G(K)G(L)>0}, an analogue of Theorem \ref{thm:Berwald} (in the range $0<p\leq q$) was proved in \cite[Thm. 1.4]{ALY}, under the condition that the concave function attains its maximum at 0:

\begin{thm}\label{thm:BerwaldDiscreteNew}
Let $K\subseteq\R^n$ be a convex body containing the origin and let $h:K\to[0,\infty)$ be a concave function. Then, for any $0<p\leq q$ we have that
$$
\left(\frac{{{n+q}\choose{n}}}{G_n(K)}\int_K h^q(x)dG_n(x)\right)^{1/q}\leq\left(\frac{{{n+p}\choose{n}}}{G_n(K)}\int_{K+C_n} (h\star\chi_{C_n})^p(x)dG_n(x)\right)^{1/p}.
$$
\end{thm}

Let us point out that, in \cite[Thms. 1.4 and 4.5]{ALY}, Theorem \ref{thm:BerwaldDiscreteNew} was proved under the additional assumption that the function $h$ attains its maximum at $0$, showing also that it implies the continuous version of Berwald's inequality, Theorem \ref{thm:Berwald}, in the range $0<p\leq q$, since in the continuous version we can assume without loss of generality that the concave function attains its maximum at 0. In this paper (see Section \ref{sec:Discrete Berwald's inequalities}) we will provide a much simpler proof of Theorem \ref{thm:BerwaldDiscreteNew} without assuming that $h$ attains its maximum at $0$, obtaining it as a consequence of Theorem \ref{thm:Berwald}. Thus, both Theorems \ref{thm:Berwald} and \ref{thm:BerwaldDiscreteNew}, are equivalent in the range $0<p\leq q$.

Let us also point out that the inclusion relation \eqref{eq:InclusionZhangRadialFunctions2} was obtained in \cite{GZ} by applying Berwald's inequality to the concave function $h: K\to[0,\infty)$ given by $h(x)=|K\cap\{x+\lambda e_n\,:\,\lambda\geq0\}|_1$ with parameters $p\to(-1)^+$ and $q=n$. The same inclusion relation was obtained in \cite[Sect. 3]{ALM} by applying Berwald's inequality to the concave function $h_1:P_{e_n^\perp}(K)\to[0,\infty)$ given by $h_1(y)=|K\cap(y+\langle e_n\rangle)|_1$ with only positive parameters $p=1$ and $q=n+1$. In the same paper \cite{ALM}, \eqref{eq:InequalityDiscreteZhang} was proved  by applying the discrete version of Berwald's inequality, Theorem \ref{thm:BerwaldDiscreteNew}, to the same function $h_1$.

In the setting of log-concave functions, the following version of Berwald's inequality was obtained in \cite[Lem. 3.3]{AAGJV} and \cite[Thm. 1.1]{ABG2}:

\begin{thm}\label{thm:FuncBerwald>-1}
Let $f\in\mathcal F(\mathbb R^n)$ and let $h:L(f)\to[0,\infty)$ be a concave function,
where $L(f)=\{(x,t)\in\R^{n+1}\,:\,f(x)\geq e^{-t}\Vert f\Vert_\infty\}$. Let $d\overline{\nu}$ be the measure defined by $d\overline{\nu}(x,t)=d\overline{\nu}_{n+1}(x,t)=dm_n(x)\otimes e^{-t}dt$. Then, for any $-1<p\leq q$,
	$$
	\left(\frac{1}{\Gamma(1+q)\overline{\nu}(L(f))}\int_{L(f)} h^q(x,t)d\overline{\nu}(x,t)\right)^{1/q}\leq\left(\frac{1}{\Gamma(1+p)\overline{\nu}(L(f))}\int_{L(f)} h^p(x,t)d\overline{\nu}(x,t)\right)^{1/p}.
	$$
\end{thm}

Theorem \ref{thm:FuncBerwald>-1} was proved in \cite{AAGJV} in the range of parameters $0<p\leq q$. It was extended to the range of parameters  $-1<p\leq q$ in \cite{ABG2}, in order to show that, when applied to the function
$h:L(f)\to[0,\infty)$ given by $h(x,t)=|L_t(f)\cap\{(x,t)+\lambda e_n\,:\,\lambda\geq0\}|_1$ with parameters $p\to(-1)^+$ and $q=n$, one obtains inequality \eqref{eq:InclusionFunctionalZhangRadialFunctions} that leads to the functional version of Zhang's inequality, \eqref{eq:FunctionalZhang}.

\begin{rmk}\label{rmk:IntegrabilityConcaveFunctionsOnEpigraphs}
Let us point out that if $f\in\mathcal{F}(\R^n)$ and $h:L(f)\to[0,\infty)$ is a concave function, then there exists an affine function $h_1:L(f)\to[0,\infty)$ of the form $h_1(x,t)=at+b$ such that $h(x,t)\leq h_1(x,t)$ for every $(x,t)\in L(f)$ and then, $h^p$ is integrable with respect to the measure $d\overline{\nu}$ for every $p>0$. Indeed, notice that the function $g:[0,\infty)\to[0,\infty)$ given by $\displaystyle{g(t)=\sup_{x\in L_t(f)}h(x,t)}$ is concave, and take $h_1(x,t)=g_1(t)$ given by a supporting line of $g$.
\end{rmk}

In this paper, we will provide a discrete version of Theorem \ref{thm:FuncBerwald>-1} (see Theorem \ref{thm:DiscreteFunctionalBerwald}), where the measure $d\overline{\nu}$ on $\R^n\times[0,\infty)$ is replaced by the measure $d\nu=d\nu_{n+1}(x,t)= dG_n(x) \otimes e^{-t}dt$, in the range $0<p\leq q$. We will see that inequality \eqref{eq:InclusionFunctionalZhangRadialFunctions} can be obtained by using Theorem \ref{thm:FuncBerwald>-1} with only positive parameters and that, in the same spirit, applying such discrete version to the appropriate function, we can obtain Theorem \ref{thm:DiscreteFunctionalZhang}.


\section{Discrete versions of Berwald's inequalities}\label{sec:Discrete Berwald's inequalities}

In this section, we provide a new proof of the discrete version of Berwald's inequality, Theorem \ref{thm:BerwaldDiscreteNew}, in the range of parameters $0<p\leq q$,  obtaining it as a consequence of the continuous version given by Theorem \ref{thm:Berwald} and showing in this way that both the discrete and the continuous versions are equivalent. Following this idea, we will also prove a discrete version of Theorem \ref{thm:FuncBerwald>-1} in the range $0<p\leq q$ (see Theorem \ref{thm:DiscreteFunctionalBerwald} below) and we will show that it implies Theorem \ref{thm:FuncBerwald>-1} in this range of the parameters.

We start by showing that the following Lemma, which is a particular case of Lemma 3.1 in \cite{Ma}, can be obtained as a consequence of Berwald's inequality:

\begin{lemma}\label{lem:LemmaBerwald}
Let $f:[0,M] \to [0, \infty)$ be a non-identically 0 decreasing concave function. For any $n \in \N$ and any $0<p\leq q$
$$
\left( \frac{\binom{n+q}{q}}{f^n(0)} \int_0^M q s^{q-1} f^n(s) ds  \right)^{1/q} \leq \left( \frac{\binom{n+p}{p}}{f^n(0)} \int_0^M p s^{p-1} f^n(s) ds \right)^{1/p}.
$$
\end{lemma}

\begin{proof} Notice that, since any concave function is continuous on the interior of its support, we can assume without loss of generality that $f$ is continuous on $[0,M]$, by changing the value at $M$ if necessary.

For any $n \in \N$ we consider any convex body $K \subseteq \R^n$ with $0 \in K$ and we define the convex body
\[
L:=\left\{ (x,s) \in \R^n \times [0,M] : x \in f(s) \frac{K}{f(0)} \right\}\subseteq\R^{n+1}.
\]

Since $f$ is a concave and decreasing function, we have that $L$ is the hypograph of the concave function $h: K \to [0,M]$ given by
\[
h(x) = \max \{ s \in \R: (x,s) \in L \}.
\]

By the continuous version of Berwald's inequality, Theorem \ref{thm:Berwald}, for any $0<p\leq q$ we have
\[
\left( \frac{\binom{n+q}{q}}{|K|} \int_K h^q(x) dx  \right)^{1/q} \leq \left( \frac{\binom{n+p}{p}}{|K|} \int_K h^p(x) dx  \right)^{1/p}.
\]

Finally notice that, for any $p>0$, taking into account that $C_s(h)=\{ x \in K: h(x) \geq s \}$ for every $s\geq0$,
\begin{eqnarray*}
\frac{1}{|K|}\int_K h^p(x) dx &=& \frac{1}{|K|}\int_0^{\infty} p s^{p-1} |C_s(h)| ds = \frac{1}{|K|}\int_0^M p s^{p-1} \left | f(s) \frac{K}{f(0)} \right | ds\cr
& = &\frac{1}{f^n(0)} \int_0^M st^{p-1} f^n(s) ds.
\end{eqnarray*}
\end{proof}

\begin{rmk}
As a matter of fact, the original proof of Berwald's inequality in \cite{Be} is obtained by proving the inequality in Lemma \ref{lem:LemmaBerwald} for the function $f(s)=|C_s(h)|^\frac{1}{n}$, where $h:K\to[0,M]$ is a concave function. In this proof of Lemma \ref{lem:LemmaBerwald} we simply constructed, given a decreasing concave function $f:[0,M]\to[0,\infty)$, the concave function $h:K\to[0,\infty)$ for some convex body $K\subseteq\R^n$ such that $f(s)=|C_s(h)|^\frac{1}{n}$.
\end{rmk}

\begin{rmk}
Let us also point out that Lemma \ref{lem:LemmaBerwald} improves inequality \eqref{eq:InequalityInclusionBallsBodies} for $\frac{1}{n}$-concave functions, i.e., functions $g:[0, \infty)\to[0,\infty)$ such that $g^\frac{1}{n}$ is concave.
\end{rmk}

Let us now prove Theorem \ref{thm:BerwaldDiscreteNew}:

\begin{proof}[Proof of Theorem \ref{thm:BerwaldDiscreteNew}]

Let $K\subseteq\R^n$ be a convex body containing the origin and $h:K\to[0,\infty)$ a concave function. Let us denote $h^\diamond: K+C_n\to[0,\infty)$ the function $h^\diamond=h\star\chi_{C_n}$.

Let us define the functions $f_1, \Tilde{f_1} \colon [0, \Vert h\Vert_\infty] \to [0, \infty)$ given by
\begin{itemize}
\item$f_1(s)=G_n(C_s(h))$ and
\item $\Tilde{f_1}(s)=G_n(C_s(h^\diamond))$.
\end{itemize}

First of all notice that since $K$ is a convex body, $h$ is bounded and for any $s_0, s_1 \in [0, \norma{h}_{\infty})$ the sets $C_{s_0}(h)$ and $C_{s_1}(h)$ are non-empty bounded sets. Thus, taking into account \eqref{eq:InclusionSuperlevelSetsAsplundConcave} and using the  concavity of $h$ and the discrete Brunn-Minkowski inequality, Theorem \ref{thm: BM_lattice_point_no_G(K)G(L)>0}, we have that for any $\lambda \in [0,1]$
\begin{align*}
\Tilde{f_1}((1-\lambda)s_0 + \lambda s_1)^{1/n} &= G_n(C_{(1-\lambda)s_0 + \lambda s_1}(h^\diamond))^{1/n} \\
&\geq G_n(C_{(1-\lambda)s_0 + \lambda s_1}(h) + C_n)^{1/n} \\
& \geq G_n((1-\lambda) C_{s_0}(h) + \lambda C_{s_1}(h) + C_n )^{1/n} \\
& \geq (1-\lambda) G_n(C_{s_0}(h))^{1/n} + \lambda G_n(C_{s_1}(h))^{1/n} \\
&=(1-\lambda)f_1(s_0)^{1/n} + \lambda f_1(s_1)^{1/n}.
\end{align*}

The above inequality also holds if $s_0=\Vert h\Vert_\infty$ or $s_1=\Vert h\Vert_\infty$ and $C_{\Vert h\Vert_\infty}(h)=\{ x \in K: h(x) = \Vert f\Vert_\infty \}\neq\emptyset$. As a consequence we have that $\textrm{conv}\left\{\text{hyp} \left (f_1^{1/n} \right) \right\} \subseteq\textrm{hyp}\left(\Tilde{f_1}^{1/n} \right)$ and then, taking  $M=\Vert h\Vert_\infty$ and $f: [0,M] \to [0,\infty)$ the concave function defined as
$$
f(s)=\sup\left\{r\geq0\,:\,(s,r)\in\textrm{conv}\left\{\textrm{hyp}\left(f_1^{1/n}\right)\right\}\right\},
$$
we have that
$$
\text{hyp} \left (f_1^{1/n} \right) \subseteq\textrm{conv}\left \{ \textrm{hyp} \left (f_1^{1/n} \right) \right \}\subseteq\textrm{hyp}(f)\subseteq \textrm{hyp}\left(\Tilde{f_1}^{1/n} \right).
$$
On the other hand, notice that $f_1$ and $\Tilde{f_1}$ are decreasing functions. Therefore, 

\begin{itemize}
	\item $f$ is concave by definition, as its hypograph is a convex set.
	\item $f$ is decreasing since $f_1$ is also decreasing.
	\item For any $s \in[0,M]$, $f_1(s) \leq f(s)^n$ since $\text{hyp} \left (f_1^{1/n} \right) \subseteq \text{hyp} (f)$.
	\item For any $s \in[0,M]$, $f(s)^n \leq \Tilde{f_1}(s)$ since $\text{hyp} (f) \subseteq \textrm{hyp}\left(\Tilde{f_1}^{1/n} \right)$.
	\item Since $f_1^{1/n}$ is non-increasing on $[0, M]$ then $f^n(0)=f_1(0)=G_n(K)$.
\end{itemize}

Therefore we can apply Lemma \ref{lem:LemmaBerwald} to the function $f$ and obtain that for any $0<p\leq q$
\[
\left( \frac{\binom{n+q}{q}}{f^n(0)} \int_0^{M} q s^{q-1} f^n(s) ds  \right)^{1/q} \leq \left( \frac{\binom{n+p}{p}}{f^n(0)} \int_0^M p s^{p-1} f^n(s) ds \right)^{1/p}.
\]

Now, by the properties of the function $f$ we have that
\[
 \left( \frac{\binom{n+q}{q}}{G_n(K)} \int_0^M q s^{q-1} f_1(s) ds  \right)^{1/q} \leq\left( \frac{\binom{n+q}{q}}{f^n(0)} \int_0^M q s^{q-1} f^n(s) ds  \right)^{1/q}
\]

and
\[
\left( \frac{\binom{n+p}{p}}{f^n(0)} \int_0^M p s^{p-1} f^n(s) ds \right)^{1/p} \leq \left( \frac{\binom{n+p}{p}}{G_n(K)} \int_0^{M} p s^{p-1} \Tilde{f_1}(s) ds \right)^{1/p}.
\]

Finally we only need to take into account that
\begin{align*}
\int_0^{M} q s^{q-1} f_1(s) ds = \int_0^{\Vert h\Vert_\infty} q s^{q-1} G_n(C_s(h)) ds = \int_{K}h^q(x)dG_n(x)
\end{align*}

and
\begin{eqnarray*}
\int_0^{M} p s^{p-1} \Tilde{f_1}(s) ds &=& \int_0^{\Vert h\Vert_\infty} q s^{q-1} G_n(C_s(h^\diamond)) ds =\int_{K+C_n}(h^\diamond)^p(x)dG_n(x)\cr
&=&\int_{K+C_n}(h\star\chi_{C_n})^p(x)dG_n(x).
\end{eqnarray*}
\end{proof}

Let us now state our discrete version of Theorem \ref{thm:FuncBerwald>-1} in the range of parameters $0<p\leq q$:

\begin{thm}\label{thm:DiscreteFunctionalBerwald}
Let $f \in \mathcal{F}(\R^n)$  such that $\norma{f}_{\infty}=f(0)$ and let  $h: L(f) \to [0, \infty)$ be a concave function, where $L(f)=\{ (x,t) \in \R^n \times [0,\infty): f(x) \geq e^{-t} \norma{f}_{\infty} \}$. Let $d\nu(x,t)=d\nu_{n+1}(x,t)=dG_n(x)\otimes e^{-t}dt$. Then, for any $0<p\leq q$ we have
\begin{eqnarray*}
&&\left( \frac{1}{\Gamma(1+q) \nu(L(f))}  \int_{L(f)} h^q(x,t) d\nu(x,t) \right)^{1/q} \cr
&\leq& \left( \frac{1}{\Gamma(1+p) \nu(L(f))}  \int_{L(f)+C_n} (h\star\chi_{C_n})^p(x,t) d\nu(x,t) \right)^{1/p}.
\end{eqnarray*}
\end{thm}

\begin{rmk}
Notice that the condition $\norma{f}_{\infty}=f(0)$ ensures that $0\in L_t(f)$ for every $t\geq0$ and then
$$
\nu(L(f))=\int_0^\infty e^{-t} G_n(L_t(f))dt\geq\int_0^\infty e^{-t}dt=1>0.
$$
\end{rmk}

Before proving Theorem \ref{thm:DiscreteFunctionalBerwald}, let us fix some notation that we will use through the rest of the section: Given $f\in\mathcal{F}(\R^n)$ such that $\norma{f}_{\infty}=f(0)$ and $h: L(f) \to [0, \infty)$ a concave function, we will denote $h^{\diamond}: L(f)+C_n\to[0,\infty)$ the function defined as $h^{\diamond}=h\star\chi_{C_n}$. Let us recall that, as in equation \eqref{eq:DefinitionC_s}, and taking into account \eqref{eq:InclusionSuperlevelSetsAsplundConcave}, for any $s\geq 0$,
\begin{itemize}
	\item $C_s(h)=\{ (x,t) \in L(f): h(x,t) \geq s \}$,
	\item $C_s(h^{\diamond})=\{ (x,t) \in L(f)+C_n: h^{\diamond}(x,t) \geq s \}\supseteq C_s(h)+C_n$.
\end{itemize}

We will also denote by $I_h, I_{h^{\diamond}} \colon [0,\infty) \to [0,\infty)$ the functions given by
\begin{align*}
I_h(s):=& \nu(C_s(h))=\int_{C_s(h)} e^{-t} dt dG_n(x)=\int_0^{\infty} e^{-t} G_n(\{ x \in \R^n: (x,t)\in C_s(h) \}) dt\cr
=& \int_0^{\infty} e^{-t} G_n(\{ x \in L(f): h(x,t) \geq s \}) dt,
\end{align*}
and, analogously,
\begin{align*}
I_{h^{\diamond}}(s):=&\nu(C_s(h^\diamond))=\int_{C_s(h^{\diamond})} e^{-t} dt dG_n(x)=\int_0^{\infty} e^{-t} G_n(\{ x \in \R^n: (x,t)\in C_s (h^{\diamond}) \}) dt \cr
=& \int_0^{\infty} e^{-t} G_n(\{ x \in L(f)+C_n: h^{\diamond}(x,t) \geq s \}) dt.
\end{align*}

Notice that, clearly, the functions $I_h$ and $I_{h^\diamond}$ are non-increasing on $[0,\infty)$.

In order to prove Theorem \ref{thm:DiscreteFunctionalBerwald} we start proving the following lemma:

\begin{lemma}\label{lem:LogConcFunctionalDiscreteBerwald}
Let $f \in \mathcal{F}(\R^n)$ and let $L(f)=\{ (x,t) \in \R^n \times [0,\infty): f(x) \geq e^{-t} \norma{f}_{\infty} \}$. Let $h \colon L(f) \to [0, \infty)$ be a concave function. For any $s_0, s_1 \in [0,\infty)$ and $\lambda \in [0,1]$, we have
\[
I_{h^{\diamond}} ((1-\lambda)s_0 +\lambda s_1) \geq I_h(s_0)^{1-\lambda} I_h(s_1)^{\lambda}.
\]
\end{lemma}

\begin{proof}
Let $s_0, s_1 \in [0, \infty)$ and $\lambda \in [0,1]$. If we fix $(x_0, t_0) \in C_{s_0}(h)$ and $(x_1, t_1) \in C_{s_1}(h)$, since $h$ is a concave function, we have that
\[
h((1-\lambda)(x_0,t_0) + \lambda (x_1,t_1)) \geq (1-\lambda) h(x_0,t_0) + \lambda h(x_1,t_1) \geq (1-\lambda) s_0 + \lambda s_1.
\]

Assume first that $C_{s_0}(h) \neq \emptyset$ and $C_{s_1}(h) \neq \emptyset$. Then,
\[
(1-\lambda)C_{s_0}(h) + \lambda C_{s_1}(h) \subseteq C_{s_{\lambda}}(h),
\]

where $s_{\lambda}=(1-\lambda)s_0 + \lambda s_1$.

As a consequence, if we fix $t_0, t_1 \in [0,\infty)$ and we denote $t_{\lambda}=(1-\lambda) t_0 + \lambda t_1$, then
\begin{align*}
(1-\lambda)\{ x \in \R^n : (x,t_0) \in C_{s_0}(h) \} + \lambda \{ x \in \R^n : (x,t_1) \in C_{s_1}(h) \} \subseteq \{ x \in \R^n : (x, t_{\lambda}) \in C_{s_{\lambda}}(h)  \}.
\end{align*}

Therefore, the set
\begin{align*}
(1-\lambda) \{ x \in \R^n : (x,t_0) \in C_{s_0}(h) \} + \lambda \{ x \in \R^n : (x,t_1) \in C_{s_1}(h) \} + C_n
\end{align*}

is contained in
\[
\{ x \in \R^n : (x, t_{\lambda}) \in C_{s_{\lambda}}(h)  \} + C_n \subseteq \{ x \in \R^n : (x, t_{\lambda}) \in C_{s_{\lambda}}(h^{\diamond})  \} .
\]

Since
\begin{itemize}
\item $\{ x \in \R^n : (x,t_0) \in C_{s_0}(h) \}\subseteq L_{t_0}(f)$ and
\item $\{ x \in \R^n : (x,t_1) \in C_{s_1}(h) \}\subseteq L_{t_1}(f)$,
\end{itemize}
we have that these sets are bounded. Therefore, if $t_0$ and $t_1$ satisfy that the previous sets are non-empty, by the discrete Brunn-Minkowski discrete inequality, Theorem \ref{thm: BM_lattice_point_no_G(K)G(L)>0}, and the arithmetic-geometric mean inequality, we have
\begin{align*}
G_n&(\{ x \in \R^n : (x, t_{\lambda}) \in C_{s_{\lambda}}(h^{\diamond}) \})^{1/n} \\
&\geq (1-\lambda) G_n(\{ x \in \R^n : (x,t_0) \in C_{s_0}(h) \} )^{1/n} +\lambda G_n( \{ x \in \R^n : (x,t_1) \in C_{s_1}(h) \} )^{1/n} \\
& \geq G_n(\{ x \in \R^n : (x,t_0) \in C_{s_0}(h) \} )^{(1-\lambda)/n} G_n( \{ x \in \R^n : (x,t_1) \in C_{s_1}(h) \} )^{\lambda/n}.
\end{align*}
Therefore, we have obtained that
\begin{align} \label{eq:BM}
G_n&(\{ x \in \R^n : (x, t_{\lambda}) \in C_{s_{\lambda}}(h^{\diamond}) \}) \cr
& \geq G_n(\{ x \in \R^n : (x,t_0) \in C_{s_0}(h) \} )^{1-\lambda} G_n( \{ x \in \R^n : (x,t_1) \in C_{s_1}(h) \} )^{\lambda}.
\end{align}
Notice that this inequality is also true if $t_0$ or $t_1$ satisfy that one of the sets on the right-hand side is the empty set.

Moreover, the exponential function satisfies that, for any $t_0,t_1\in[0,\infty)$ and any $\lambda\in[0,1]$,
\[
e^{-t_{\lambda}} = e^{-((1-\lambda)t_0+ \lambda t_1)} = (e^{-t_0})^{1-\lambda} (e^{-t_1})^{\lambda}.
\]
Therefore, inequality $\eqref{eq:BM}$ implies that for any fixed $s_0,s_1 \in [0, \infty)$ and $\lambda \in [0,1]$ and any $t_0, t_1 \in [0, \infty)$
\begin{align*}
e^{-t_{\lambda}} & G_n(\{ x \in \R^n : (x, t_{\lambda}) \in C_{s_{\lambda}}(h^{\diamond}) \}) \\
& \geq (e^{-t_0} G_n(\{ x \in \R^n : (x,t_0) \in C_{s_0}(h) \} ) )^{1-\lambda} (e^{-t_1} G_n( \{ x \in \R^n : (x,t_1) \in C_{s_1}(h) \} ))^{\lambda}.
\end{align*}

Applying Pr\'ekopa-Leindler's inequality (see, for instance, \cite[Thm. 1.4.1]{AGM}) in $\R$ to the functions
\begin{align*}
f(t) &=e^{-t} G_n(\{ x \in \R^n : (x,t) \in C_{s_0}(h) \} ) \chi_{[0, \infty)}(t), \\
g(t)&= e^{-t} G_n(\{ x \in \R^n : (x,t) \in C_{s_1}(h) \} ) \chi_{[0, \infty)}(t), \textrm{ and} \\
h(t)&=e^{-t} G_n(\{ x \in \R^n : (x, t) \in C_{s_{\lambda}}(h^{\diamond}) \}) \chi_{[0, \infty)}(t),
\end{align*}
we obtain that for any fixed $s_0,s_1 \in [0, \infty)$ and $\lambda \in [0,1]$
\begin{align*}
\int_0^{\infty} e^{-t} G_n(\{ x \in \R^n : (x, t) \in C_{s_{\lambda}}(h^{\diamond}) \}) dt&\geq \left( \int_0^{\infty} e^{-t} G_n(\{ x \in \R^n : (x,t) \in C_{s_0}(h) \} ) dt \right)^{1-\lambda} \cr
&\times\left( \int_0^{\infty} e^{-t} G_n(\{ x \in \R^n : (x,t) \in C_{s_1}(h) \} ) dt \right)^{\lambda}.
\end{align*}
Equivalently,
\[
I_{h^{\diamond}} ((1-\lambda)s_0 + \lambda s_1) \geq I_h(s_0)^{1-\lambda} I_h(s_1)^{\lambda}.
\]
\end{proof}

Let us now prove Theorem \ref{thm:DiscreteFunctionalBerwald}

\begin{proof}[Proof of  Theorem \ref{thm:DiscreteFunctionalBerwald}]

By Lemma \ref{lem:LogConcFunctionalDiscreteBerwald}, we have that for any $s_0,s_1 \in [0, \infty)$ and $\lambda\in [0,1]$
$$
I_{h^{\diamond}} ((1-\lambda)s_0 + \lambda s_1) \geq I_h(s_0)^{1-\lambda} I_h(s_1)^{\lambda}.
$$
Equivalently, understanding $\log 0=-\infty$, for any $s_0,s_1 \in [0, \infty)$ and $\lambda \in [0,1]$,
$$
\log I_{h^{\diamond}} ((1-\lambda)s_0 + \lambda s_1) \geq (1-\lambda) \log I_h(s_0) + \lambda \log I_h(s_1).
$$
Thus, $\textrm{conv}\left\{\textrm{epi}(-\log(I_h)\right\}\subseteq\textrm{epi}(-\log(I_h^\diamond))$.

\medskip

Let $u:[0,\infty) \to (-\infty, \infty]$ be the convex function defined as
$$
u(s)=\begin{cases}\textrm{inf}\left\{r\in\R\,:\,(s,r)\in\textrm{conv}\left\{\textrm{epi}(-\log  I_h)\right\}\right\} &\textrm{ if } -\log  I_h(s)\neq\infty\\
\infty&\textrm{ if } -\log  I_h(s)=\infty.
\end{cases}
$$
and notice that
$$
\textrm{epi}(-\log(I_h))\subseteq\textrm{conv}\left\{\textrm{epi}(-\log(I_h)\right\}\subseteq\textrm{epi}(u)\subseteq\textrm{epi}(-\log(I_h^\diamond)).
$$
Then for any $s \geq 0$
$$
-\log I_{h^{\diamond}}(s) \leq u(s) \leq -\log  I_h(s).
$$

Equivalently, calling $g:[0,\infty)\to[0,\infty)$ the log-concave function $g(s)=e^{-u(s)}$, we have that for every $s\geq 0$
$$
I_h(s)\leq g(s)\leq I_{h^\diamond}(s).
$$

Notice that, since $I_h$ is non-increasing on $[0, \infty)$, then $-\log I_h$ is non-decreasing on $[0, \infty)$ and then $u(0)=-\log I_h(0)$. Thus, $g(0)=I_h(0)=\nu(L(f))>0$.

By \eqref{eq:InequalityInclusionBallsBodies}, the function $g$ satisfies that, for any $0<p\leq q$,
\begin{align*}
\left( \frac{1}{\Gamma(1+q) g(0)}  \int_0^{\infty}qs^{q-1}g(s) ds \right)^{1/q} \leq \left( \frac{1}{\Gamma(1+p) g(0) }  \int_0^{\infty} ps^{p-1}g(s) ds \right)^{1/p}.
\end{align*}

Since $g(0)=\nu(L(f))$ and $I_h(s) \leq g(s) \leq I_{h^{\diamond}}(s)$ for every $s\geq0$, we obtain that
\begin{align*}
\left( \frac{1}{\Gamma(1+q) \nu(L(f))}  \int_0^{\infty}qs^{q-1}I_h(s) ds \right)^{1/q} \leq \left( \frac{1}{\Gamma(1+p) \nu(L(f)) }  \int_0^{\infty} ps^{p-1} I_{h^{\diamond}}(s) ds \right)^{1/p}.
\end{align*}

Finally we only have to notice that, using Fubini's Theorem,
$$
\int_0^{\infty}qs^{q-1}I_h(s) ds = \int_0^{\infty} q s^{q-1} \nu(\{ (x,t) \in L(f) : h(x,t) \geq s  \}) ds = \int_{L(f)} h^q(x,t) d\nu(x,t)
$$
and, in the same way,
\begin{eqnarray*}
\int_0^{\infty}ps^{p-1}I_{h^{\diamond}}(s) ds &=& \int_0^{\infty} p s^{p-1} \nu(\{ (x,t) \in L(f)+C_n : h^{\diamond}(x,t) \geq s  \}) ds \cr
&=& \int_{L(f)+C_n} (h^{\diamond})^p (x,t) d\nu(x,t).
\end{eqnarray*}
\end{proof}

\subsection{From the discrete to the continuous Berwald's inequality}

In this subsection we are going to show that Theorem \ref{thm:DiscreteFunctionalBerwald} implies Theorem \ref{thm:FuncBerwald>-1} in the range of parameters $0<p\leq q$.

\begin{proof}[Proof of Theorem \ref{thm:DiscreteFunctionalBerwald} implies Theorem \ref{thm:FuncBerwald>-1}]

Let $f=\Vert f\Vert_\infty e^{-u}\in\mathcal{F}(\R^n)$ and $h:L(f)\to[0,\infty)$ a concave function. Notice that we can assume, without loss of generality, that $\Vert f\Vert_\infty=f(0)$. Since any concave function is continuous on the interior of its support and  every $f\in\mathcal{F}(\R^n)$ is continuous on the interior of its support, taking into account \eqref{eq:barnu(L)Integral} and that $\overline{\nu}(\partial L(f))=0$, we can also assume, without loss of generality, that $L(f)$ is closed and $h$ is continuous on $L(f)$. Otherwise, change $f$ to $f_1=\Vert f\Vert_\infty e^{-u_1}$, the log-concave function such that $\textrm{epi}(u_1)=\overline{\textrm{epi}(u)}$, and change $h$ to $h_1:L(f_1) \to [0,\infty)$, the concave function whose hypograph is $\overline{\textrm{hyp}(h)}$, which we will rename as $f$ and $h$.

Let us define, for any $\lambda>0$, the function $f_\lambda\in\mathcal{F}(\R^n)$ given by $f_\lambda(x)=f\left( \frac{x}{\lambda} \right)$. let us recall that $\norma{f_\lambda}_{\infty}=\norma{f}_{\infty}$ and, for any $t \geq 0$, $L_t(f_\lambda)=\lambda L_t(f)$.
Thus,
$$
L(f_\lambda)=\{ (x,t) : f_\lambda(x) \geq e^{-t} \norma{f_\lambda}_{\infty} \} =\{(\lambda x,t)\,:\,(x,t)\in L(f)\}.
$$
Let us also define the concave function $h_\lambda \colon L(f_\lambda) \to [0, \infty)$ given by $h_\lambda(x,t)=h\left(\frac{x}{\lambda}, t \right)$.

For any $0<p\leq q$, we apply Theorem \ref{thm:DiscreteFunctionalBerwald} to the function $h_\lambda$ and we obtain that
\begin{align*}
&\left( \frac{1}{\Gamma(1+q) \nu(L(f_\lambda))}  \int_{L(f_\lambda)} h_\lambda^q(x,t) d\nu(x,t) \right)^{1/q}  \cr
&\leq\left( \frac{1}{\Gamma(1+p) \nu(L(f_\lambda))}  \int_{L(f_\lambda)+C_n} (h_\lambda\star\chi_{C_n})^p(x,t) d\nu(x,t) \right)^{1/p}.
\end{align*}

First of all notice that by Lemma \ref{lem:LogConcaveIntegralDiscreteToContinuous}, \eqref{eq:nu(L)Integral} and \eqref{eq:barnu(L)Integral}
\begin{equation}\label{set_L_r}
\lim_{\lambda \to \infty} \frac{\nu(L(f_\lambda))}{\lambda^n} =\lim_{\lambda \to \infty}\frac{1}{\lambda^n}\int_{\R^n}\frac{f_\lambda(x)}{\Vert f\Vert_\infty}dG_n(x) = \int_{\R^n}\frac{f(x)}{\Vert f\Vert_\infty}dx=\overline{\nu}(L(f)).
\end{equation}

Let us now prove that
\begin{equation}\label{eq:equalityexponentNew}
\lim_{\lambda \to \infty} \frac{1}{\lambda^n} \int_{L(f_\lambda)}h_\lambda^q(x,t) d\nu(x,t)=\int_{L(f)} h^q(x,t)d\overline{\nu}(x,t).
\end{equation}
Equivalently, let us prove that
\begin{equation}\label{equality_exponent_q}
\lim_{\lambda \to \infty} \frac{1}{\lambda^n} \int_0^{\infty} e^{-t} \int_{\lambda L_t(f)} h\left(\frac{x}{\lambda},t \right)^q dG_n(x) dt = \int_0^\infty e^{-t}\int_{L_t(f)} e^{-t} h(x,t)^q dxdt.
\end{equation}

On the one hand, since for any $t\geq0$ the function $h(\cdot, t)$ is integrable on $L_t(f)$, by \eqref{Riemann_integrability} we have that for every $t\geq0$
$$
\lim_{\lambda \to \infty} \frac{e^{-t}}{\lambda^n}  \int_{\lambda L_t(f)} h\left(\frac{x}{\lambda},t \right)^q dG_n(x) =e^{-t} \int_{L_t(f)} h(x,t) dx.
$$

On the other hand,  for every $t\geq0$ and every $\lambda\geq\frac{1}{2}$, we have by \eqref{eq:LatticePointAndVolume} and \eqref{eq:InclusionSuperlevelSetsAsplundConcave}
\begin{align*}
\frac{e^{-t}}{\lambda^n}  \int_{\lambda L_t(f)} h\left(\frac{x}{\lambda},t \right)^q dG_n(x) &= \frac{e^{-t}}{\lambda^n} \int_0^{\infty} q s^{q-1} G_n\left(\left\{ x \in \lambda L_t(f) : h\left(\frac{x}{\lambda},t \right) \geq s \right\}\right) ds \\
&= \frac{e^{-t}}{\lambda^n} \int_0^{\infty} q s^{q-1} G_n\left( \lambda \left\{ x \in L_t(f) : h\left(x,t \right) \geq s \right\}\right) ds \\
& \leq  \frac{e^{-t}}{\lambda^n} \int_0^{\infty} q s^{q-1} \left| \lambda \left\{ x \in L_t(f) : h\left(x,t \right) \geq s \right\} + \frac{1}{2} B_{\infty}^n \right| ds\\
& \leq  \frac{e^{-t}}{\lambda^n} \int_0^{\infty} q s^{q-1} \left| \lambda \left\{ x \in L_t(f) : h\left(x,t \right) \geq s \right\} + \lambda B_{\infty}^n \right| ds \\
&= e^{-t} \int_0^{\infty} q s^{q-1} \left|  \left\{ x \in L_t(f) : h\left(x,t \right) \geq s \right\} + B_{\infty}^n \right| ds \\
& \leq e^{-t} \int_0^{\infty} q s^{q-1} \left|  \left\{ x \in L_t(f)+B_\infty^n : h\star\chi_{B_\infty^n\times\{0\}}(x,t) \geq s \right\}  \right| ds \\
&=e^{-t}\int_{L_t(f)+B_\infty^n}(h\star\chi_{B_\infty^n\times\{0\}})^q(x,t) dx,
\end{align*}
which is integrable on $t \in [0, \infty)$ since $h\star\chi_{B_\infty^n\times\{0\}}:L(f)+(B_\infty^n\times\{0\})\to[0,\infty)$ is concave on $L(f)+(B_\infty^n\times\{0\})$ and, therefore, has moments of all orders with respect to the measure $d\overline{\nu}$ (see Remark \ref{rmk:IntegrabilityConcaveFunctionsOnEpigraphs}).

As a consequence, by the dominated convergence theorem we obtain \eqref{equality_exponent_q}.

Let now denote, for any $\lambda>0$, $h_\lambda^{\diamond}: L(f_\lambda)+C_n\to[0,\infty)$, the function given by $h_\lambda^\diamond=h_\lambda\star\chi_{C_n}$. Let us fix $\varepsilon>0$ and notice that, for every $t\geq0$ and for every $\lambda >\frac{1}{\varepsilon}$, we have

\begin{align*}
\int_{\lambda L_t(f)+C_n} (h_\lambda^{\diamond})^p(x,t) dG_n(x) &= \sum_{x \in (\lambda L_t(f) + C_n) \cap \Z^n} (h_\lambda^{\diamond})^p (x,t) \\
&= \sum_{\frac{x}{\lambda} \in \left\{ L_t(f) + \frac{1}{\lambda} C_n \right\} \cap \left\{ \frac{1}{\lambda} \Z^n \right\}}\left( \sup_{u \in C_n\atop \frac{x+u}{\lambda}\in L_t(f)} h\left( \frac{x+u}{\lambda}, t \right) \right)^p \\
&=\sum_{y \in \left\{ L_t(f) + \frac{1}{\lambda} C_n \right\} \cap \left\{ \frac{1}{\lambda} \Z^n \right\}}\left( \sup_{v \in \frac{1}{\lambda} C_n\atop y+v \in L_t(f)} h\left( y+v, t \right) \right)^p \\
& \leq \sum_{y \in \left\{ L_t(f) +\epsilon C_n \right\} \cap \left\{ \frac{1}{\lambda} \Z^n \right\}} \left( \sup_{v \in \epsilon C_n\atop y+v\in L_t(f)} h\left( y+v, t \right) \right)^p \\
& \leq \sum_{y \in \left\{ L_t(f) +\epsilon C_n \right\} \cap \left\{ \frac{1}{r} \Z^n \right\}} (h^{\diamond_{\epsilon}})^p (y,t),
\end{align*}
where $h^{\diamond_{\epsilon}}:L(f)+\epsilon C_n \to [0,\infty)$ is given by $h^{\diamond_{\epsilon}}=h\star\chi_{\varepsilon C_n}$.

Since for any $t\geq0$ the function $h^{\diamond_{\epsilon}}(\cdot,t)$ is integrable on $L_t(f)+\epsilon C_n$, by \eqref{Riemann_integrability} we have that for any $t\geq0$
\[
\lim_{\lambda \to \infty} \frac{e^{-t}}{\lambda^n} \sum_{y \in \left\{ L_t(f) +\epsilon C_n \right\} \cap \left\{ \frac{1}{\lambda} \Z^n \right\}} (h^{\diamond_{\epsilon}})^p(y,t) = e^{-t} \int_{L_t(f) + \epsilon C_n} (h^{\diamond_{\epsilon}})^p (x,t) dx.
\]

Moreover, for every $t\geq0$ and every $\lambda\geq\max\left\{\frac{1}{2},\frac{1}{\epsilon}\right\}$, using \eqref{eq:LatticePointAndVolume} and \eqref{eq:InclusionSuperlevelSetsAsplundConcave}, we have
\begin{align*}
\frac{e^{-t}}{\lambda^n} \sum_{y \in \left\{ L_t(f) +\epsilon C_n \right\} \cap \left\{ \frac{1}{\lambda} \Z^n \right\}} (h^{\diamond_{\epsilon}})^p (y,t) &= \frac{e^{-t}}{\lambda^n} \int_0^{\infty} p s^{p-1} G_n \left( \lambda \left\{ x \in L_t(f)+\epsilon C_n : h^{\diamond_{\epsilon}} (x,t) \geq s \right\}  \right) ds \\
& \leq  \frac{e^{-t}}{\lambda^n} \int_0^{\infty} p s^{p-1}  \left| \lambda \left\{ x \in L_t(f)+\epsilon C_n : h^{\diamond_{\epsilon}} (x,t) \geq s \right\} + \frac{1}{2} B_{\infty}^n \right| ds \\
&\leq \frac{e^{-t}}{\lambda^n} \int_0^{\infty} p s^{p-1}  \left| \lambda \left\{ x \in L_t(f)+\epsilon C_n : h^{\diamond_{\epsilon}} (x,t) \geq s \right\} + \lambda B_{\infty}^n \right| ds \\
&= e^{-t} \int_0^{\infty} p s^{p-1}  \left|  \left\{ x \in L_t(f)+\epsilon C_n : h^{\diamond_{\epsilon}} (x,t) \geq s \right\} + B_{\infty}^n \right| ds \\
&\leq e^{-t} \int_0^{\infty} p s^{p-1}  \left|  \left\{ x \in L_t(f)+\epsilon C_n : h^{\diamond_{\epsilon}} \star \chi_{B_{\infty}^n\times\{0\}} (x,t) \geq s \right\}  \right| ds \\
&= e^{-t} \int_{L_t(f)+\epsilon C_n} (h^{\diamond_{\epsilon}} \star \chi_{B_{\infty}^n\times\{0\}})^p (x,t)dx,
\end{align*}
which is integrable on $t\in [0, \infty)$ since since $h^{\diamond_{\epsilon}} \star \chi_{B_{\infty}^n\times\{0\}}:L(f)+\epsilon C_n+(B_\infty^n\times\{0\})\to[0,\infty)$ is concave on $L(f)+\epsilon C_n+(B_\infty^n\times\{0\})$ and, therefore, has moments of all orders with respect to the measure $d\overline{\nu}$ (see Remark \ref{rmk:IntegrabilityConcaveFunctionsOnEpigraphs}).

Thus, by the dominated convergence theorem, we conclude that
\begin{equation} \label{h_diamond}
\lim_{\lambda \to \infty} \frac{1}{\lambda^n} \int_0^{\infty} e^{-t} \sum_{y \in \left\{ L_t(f) +\epsilon C_n \right\} \cap \left\{ \frac{1}{\lambda} \Z^n \right\}} (h^{\diamond_{\epsilon}})^p(y,t) dt = \int_0^{\infty} e^{-t} \int_{L_t(f) + \epsilon C_n} (h^{\diamond_{\epsilon}})^p (x,t) dx dt.
\end{equation}

Thus, for any $\epsilon>0$ and any $\lambda>\frac{1}{\epsilon}$  we have that
\begin{align*}
&\left( \frac{1}{\Gamma(1+q) \frac{1}{\lambda^n} \nu(L(f_r)) } \frac{1}{\lambda^n} \int_{L(f_r)}h^q(x,t) d\nu(x,t) \right)^{1/q} \\
&\leq \left( \frac{1}{\Gamma(1+p) \frac{1}{\lambda^n} \nu(L(f_r))} \frac{1}{\lambda^n} \int_0^{\infty} e^{-t} \left(\sum_{y \in \left\{ L_t(f) +\epsilon C_n \right\} \cap \left\{ \frac{1}{\lambda} \Z^n \right\}} (h^{\diamond_{\epsilon}})^p(y,t)\right) dt \right)^{1/p}.
\end{align*}

Taking limits when $\lambda \to \infty$ and applying \eqref{set_L_r}, \eqref{eq:equalityexponentNew} and \eqref{h_diamond} we get
\begin{align*}
&\left( \frac{1}{\Gamma(1+q) \overline{\nu}(L(f))} \int_{L(f)} h^q(x,t) d\overline{\nu}(x,t) \right)^{1/q} \\
&\leq \left( \frac{1}{\Gamma(1+p) \overline{\nu}(L(f))} \int_0^{\infty} e^{-t} \int_{L_t(f)+\epsilon C_n} (h^{\diamond_{\epsilon}})^p(x,t) dx dt \right)^{1/p}.
\end{align*}

Since $\epsilon>0$ is arbitrary, to conclude the proof it is enough to show that
\begin{equation}\label{eq:infimumBerwald}
\inf_{\epsilon>0} \int_0^{\infty} e^{-t} \int_{L_t(f)+\epsilon C_n} (h^{\diamond_{\epsilon}})^p(x,t) dxdt \leq \int_{L(f)}h^p(x,t)d\overline{\nu}(x,t).
\end{equation}

We observe that for any $\varepsilon>0$, by \eqref{eq:InverseInclusionSuperlevelSetsAsplundConcave} and the fact that $h$ is continuous on the closed set $L(f)$,
\begin{eqnarray*}
\{ x \in L_t(f) + \epsilon C_n : h^{\diamond_{\epsilon}}(x,t) \geq s \} &\subseteq& \bigcap_{\varepsilon_1>0}\overline{\{ x \in L_t(f)  : h(x,t) \geq s-\varepsilon_1 \}} +  \epsilon B_{\infty}^n\cr
&=&\bigcap_{\varepsilon_1>0}\{ x \in L_t(f)  : h(x,t) \geq s-\varepsilon_1 \} +  \epsilon B_{\infty}^n\cr
&=&\{ x \in L_t(f)  : h(x,t) \geq s \} +  \epsilon B_{\infty}^n,
\end{eqnarray*}
so taking volumes $|\{ x \in L_t(f) + \epsilon C_n : h^{\diamond_{\epsilon}}(x,t)   \geq s \}| \leq |\{ x \in L_t(f)  : h(x,t) \geq s \} +  \epsilon B_{\infty}^n|$.

Thus,
\begin{eqnarray*}
\int_0^{\infty} e^{-t} \int_{L_t(f)+\epsilon C_n} (h^{\diamond_{\epsilon}})^p (x,t) dxdt &=& \int_0^{\infty} e^{-t} \int_0^{\infty} ps^{p-1} |\{ x \in L_t(f) + \epsilon C_n : h^{\diamond_{\epsilon}}(x,t) \geq s \}| ds dt\cr
&\leq&\int_0^{\infty} e^{-t} \int_0^{\infty} ps^{p-1} |\{ x \in L_t(f)  : h(x,t) \geq s \} +\epsilon B_\infty^n| ds dt.
\end{eqnarray*}

Taking a decreasing sequence $(\epsilon_n)_{n=1}^\infty$ converging to $0$ and using the monotone convergence theorem, we obtain
\begin{eqnarray*}
\int_0^{\infty} e^{-t} \int_{L_t(f)+\epsilon C_n} (h^{\diamond_{\epsilon}})^p (x,t) dxdt&\leq&\int_0^{\infty} e^{-t} \int_0^{\infty} ps^{p-1} |\{ x \in L_t(f)  : h(x,t) \geq s \}| ds dt\cr
&=&\int_0^\infty e^{-t}\int_{L_t(f)}h^p(x,t)dxdt=\int_{L(f)}h^p(x,t)d\overline{\nu}(x,t),
\end{eqnarray*}
which proves \eqref{eq:infimumBerwald}.
\end{proof}


\section{Functional Zhang's inequality}\label{sec:FunctionalZhang}

The main purpose of this section is to prove Theorem \ref{thm:DiscreteFunctionalZhang}. We will show that, in the same way as in \cite{ALM}, where inequality \eqref{eq:InequalityDiscreteZhang} was obtained from the discrete version of Berwald's inequality (Theorem \ref{thm:BerwaldDiscreteNew}) in the range $0<p\leq q$, we can obtain Theorem \ref{thm:DiscreteFunctionalZhang} from the discrete version of the functional Berwald's inequality we proved in Theorem \ref{thm:DiscreteFunctionalBerwald}.

We will begin this section providing a new proof of inequality \eqref{eq:InclusionFunctionalZhangRadialFunctions}, obtaining it from the functional Berwald's inequality given by Theorem \ref{thm:FuncBerwald>-1} in the range $0<p\leq q$. We will later use the same approach in the discrete setting in order to prove Theorem \ref{thm:DiscreteFunctionalZhang}.

\subsection{New proof of the functional Zhang's inequality}\label{subsec:NewProof}

In this subsection, we will provide a new proof of inequality \eqref{eq:InclusionFunctionalZhangRadialFunctions}, which is equivalent to \eqref{eq:InclusionFunctionalZhang}, and therefore implies the functional version of Zhang's inequality given by \eqref{eq:FunctionalZhang}. The main difference between this new proof and the ones in \cite{ABG1} and \cite{ABG2}, is that we will only make use of Berwald's inequality, Theorem \ref{thm:FuncBerwald>-1},  in the range $0 < p \leq q$. Since a version of Berwald's inequality in this range of the parameters has been provided in the discrete setting in Theorem \ref{thm:DiscreteFunctionalBerwald}, we will be able to use the same approach in order to prove
Theorem \ref{thm:DiscreteFunctionalZhang}.

We will make use of the following lemma, which is a direct consequence of \cite[Lem. 3.3]{ABG2}. We will write its proof here for the sake of completeness.

\begin{lemma}\label{lem:LemmaCovariogramRadialFunction}
Let $f \in \mathcal{F}(\R^n)$ and let $g_f$ be the covariogram function of $f$. Then, for any $p>0$
$$
p\int_0^\infty r^{p-1}g_f(re_n)dr=\frac{1}{p+1}\int_{L(P_{e_n^\perp}f)}|L_t(f)\cap(y+\langle e_n)\rangle|_1^{p+1}d\overline{\nu}_{n}(y,t).
$$
\end{lemma}

\begin{proof}
Taking into account \eqref{eq:Covariogram} and the fact that for any convex body $K\subseteq\R^n$, its covariogram function is supported on the difference body $K-K$,
\begin{eqnarray*}
&&p\int_0^{\infty} r^{p-1} g_f(re_n) dr=\int_0^{\infty} p r^{p-1} \int_0^{\infty} e^{-t} |L_t(f) \cap (re_n+L_t(f)) | dt dr \cr
&=& \int_0^{\infty} e^{-t} \int_0^{\rho_{L_t(f)-L_t(f)}(e_n)} p r^{p-1}  |L_t(f) \cap (re_n+L_t(f)) | dr dt \cr
&=& \int_0^{\infty} e^{-t} \int_0^{\rho_{L_t(f)-L_t(f)}(e_n)} p r^{p-1}  \int_{P_{e_n^{\perp}}L_t(f)} \max \{|L_t(f) \cap (y+\langle e_n \rangle) |_1 - r , 0  \} dy dr dt \cr
&=&  \int_0^{\infty} e^{-t} \int_{P_{e_n^{\perp}}L_t(f)} \int_0^{ |L_t(f) \cap (y+\langle e_n \rangle) |_1  } p r^{p-1}  ( |L_t(f) \cap (y+\langle e_n \rangle) |_1 - r ) dr dy dt \cr
&=&\frac{1}{p+1}\int_0^{\infty} e^{-t} \int_{P_{e_n^{\perp}}L_t(f)}  |L_t(f) \cap (y+\langle e_n \rangle) |_1^{p+1} dy dt\cr
&=&\frac{1}{p+1}\int_{L(P_{e_n^\perp}f)}|L_t(f)\cap(y+\langle e_n)\rangle|_1^{p+1}d\overline{\nu}_{n}(y,t).
\end{eqnarray*}
\end{proof}

We proceed now with our new proof of \eqref{eq:InclusionFunctionalZhangRadialFunctions}, which follows the lines of the proof of \eqref{eq:InclusionZhangRadialFunctions2} provided in \cite{ALM}.

\begin{proof}[New proof of inequality \eqref{eq:InclusionFunctionalZhangRadialFunctions}]
Let $f\in\mathcal{F}(\R^n)$ and let $L(f)\subseteq\R^n\times[0,\infty)$ and $L_t(f)\subseteq\R^n$ for any $t\geq0$ be defined as in Section \ref{sec:LogconcaveFunctions}. Taking into account that $P_{e_n^\perp}(L(f))=L(P_{e_n^\perp}f)\subseteq e_n^\perp\times[0,\infty)$, we define the function $h:L(P_{e_n^\perp}f)\to[0,\infty)$ given by
$$
h(y,t)=|L(f) \cap ((y,t)+\langle e_n \rangle )|_1=|L_t(f) \cap (y+\langle e_n \rangle)|_1,
$$
for every $(y,t)\in\R^{n-1}\times[0,\infty)$ such that $y\in P_{e_n^\perp}(L_t(f))$. Since $L(f)$ is convex and $L_t(f)$ is a bounded convex set for any $t\geq0$, $h$ is well-defined and is a concave function. Therefore, we can apply the functional version of Berwald's inequality, Theorem \ref{thm:FuncBerwald>-1} on $\tilde{L}=L(P_{e_n^\perp}f)=P_{e_n^\perp}(L(f))$ with parameters $p=1<n+1=q$ and obtain
$$
	\left(\frac{1}{\Gamma(n+2)\overline{\nu}_n(\tilde{L})}\int_{\tilde{L}} h^{n+1}(y,t)d\overline{\nu}_n(y,t)\right)^{1/(n+1)}\leq\frac{1}{\Gamma(2)\overline{\nu}_n(\tilde{L})}\int_{\tilde{L}} h(y,t)d\overline{\nu}_n(y,t),
$$
where $d\overline{\nu}_n=dm_{n-1}\otimes e^{-t}dt$.

Equivalently,
$$
\frac{1}{n+1}\int_{\tilde{L}} h^{n+1}(y,t)d\overline{\nu}_n(y,t)\leq \frac{n!}{\overline{\nu}_n(\tilde{L})^n}\left(\int_{\tilde{L}} h(y,t)d\overline{\nu}_n(y,t)\right)^{n+1}.
$$

On the one hand, by Lemma \ref{lem:LemmaCovariogramRadialFunction}
\begin{eqnarray*}
\frac{1}{n+1}\int_{\tilde{L}} h^{n+1}(y,t)d\overline{\nu}_n(y,t)&=&n\int_0^\infty r^{n-1}g_f(re_n)dr\cr
&=&n\int_0^\infty r^{n-1}\int_{\R^n}\min\left\{\frac{f(z)}{\Vert f\Vert_\infty},\frac{f(z-re_n)}{\Vert f\Vert_\infty}\right\}dzdr.
\end{eqnarray*}

On the other hand, since $\tilde{L}=L(P_{e_n^\perp}f)$, by \eqref{eq:barnu(L)Integral} we have
$$
\overline{\nu}_n(\tilde{L})=\int_{e_n^\perp}\frac{P_{e_n^\perp}f (y)}{\Vert P_{e_n^\perp}f\Vert_\infty}dy=\int_{e_n^\perp}\frac{P_{e_n^\perp}f (y)}{\Vert f\Vert_\infty}dy.
$$

Finally, by the definition of $h$,  Fubini's theorem and by \eqref{eq:barnu(L)Integral}
\begin{eqnarray*}
\int_{\tilde{L}} h(y,t)d\overline{\nu_n}(y,t)&=&\int_{0}^\infty e^{-t}\int_{P_{e_n^\perp}L_t(f)}|L_t(f) \cap (y+\langle e_n \rangle )|_1dydt\cr
&=&\int_{0}^\infty e^{-t}|L_t(f)|dt=\overline{\nu}_{n+1}(L(f))=\int_{\R^n}\frac{f(x)}{\Vert f\Vert_\infty}dx.
\end{eqnarray*}
Therefore,
$$
n\int_0^\infty r^{n-1}\int_{\R^n}\min\{f(z),f(z-re_n)\}dzdr\leq n!\frac{\left(\int_{\R^n}f(x)dx\right)^{n+1}}{\left(\int_{e_n^\perp}P_{e_n^\perp}f(y)dy\right)^n}.
$$
\end{proof}



\subsection{Discrete functional Zhang's inequality}




In this subsection we are going to prove Theorem \ref{thm:DiscreteFunctionalZhang}. In order to do that, given $f\in\mathcal{F}(\R^n)$, we define $g_f^\mu:\R^n\to[0,\infty)$ as the following variant of the covariogram function of $f$:
$$
g_f^{\mu}(x):=\int_{\R^n}\min\left\{\frac{f(z)}{\Vert f\Vert_\infty},\frac{f(z-x)}{\Vert f\Vert_\infty}\right\}d\mu_n(z).
$$

The function $g_f^\mu$ is well defined since, for any $x$, the function given by $\tilde{f}(z)=\min\left\{\frac{f(z)}{\Vert f\Vert_\infty},\frac{f(z-x)}{\Vert f\Vert_\infty}\right\}$ satisfies that $\tilde{f}\in\mathcal{F}(\R^n)$ with integral $g_f(x)$. Therefore, by \eqref{eq:integralMuandm}, $\tilde{f}$ is integrable with respect to $d\mu_n$. Besides, notice that for every $x\in\R^n$ we have
\begin{equation}\label{eq:CovariogramMu}
g_f^{\mu}(x)=\int_0^\infty e^{-t}\mu_n\left(L_t(f)\cap(x+L_t(f))\right)dt=\sigma_{n+1}\left(L(f)\cap((x,0)+L(f)\right).
\end{equation}
Notice also that $g_f^{\mu}(x)$ satisfies that
$$
\Vert g_f^\mu\Vert_\infty=g_f^\mu(0)=\int_{\R^n}\frac{f(z)}{\Vert f\Vert_\infty}d\mu_n(z).
$$

We will make use of the following lemma, which is analogue to Lemma \ref{lem:LemmaCovariogramRadialFunction} and whose proof follows the same lines.

\begin{lemma}\label{lem:IdentityDiscrete}
Let $f\in\mathcal{F}(\R^n)$ and $p>0$. Then
$$
p \int_0^{\infty} r^{p-1}g_f^\mu(re_n)dr=\frac{1}{p+1}\int_{L(P_{e_n^\perp}f)}|L_t(f)\cap(y+\langle e_n\rangle)|_1^{n+1}d\nu_n(y,t).
$$
\end{lemma}

\begin{proof}
Notice that, by \eqref{eq:CovariogramMu},
\begin{eqnarray*}
&&p\int_0^{\infty} r^{p-1} g_f^\mu(re_n) dr=\int_0^{\infty} p r^{p-1} \int_0^{\infty} e^{-t} \mu_n(L_t(f) \cap (re_n+L_t(f))) dt dr \cr
&=& \int_0^{\infty} e^{-t} \int_0^{\rho_{L_t(f)-L_t(f)}(e_n)} p r^{p-1}  \mu_n(L_t(f) \cap (re_n+L_t(f))) dr dt \cr
&=& \int_0^{\infty} e^{-t} \int_0^{\rho_{L_t(f)-L_t(f)}(e_n)} p r^{p-1}  \int_{P_{e_n^{\perp}}L_t(f)} \max \{|L_t(f) \cap (y+\langle e_n \rangle) |_1 - r , 0  \} dG_{n-1}(y) dr dt \cr
&=&  \int_0^{\infty} e^{-t} \int_{P_{e_n^{\perp}}L_t(f)} \int_0^{ |L_t(f) \cap (y+\langle e_n \rangle) |_1  } p r^{p-1}  ( |L_t(f) \cap (y+\langle e_n \rangle) |_1 - r ) dr dG_{n-1}(y) dt \cr
&=&\frac{1}{p+1}\int_0^{\infty} e^{-t} \int_{P_{e_n^{\perp}}L_t(f)}  |L_t(f) \cap (y+\langle e_n \rangle) |_1^{p+1} dG_{n-1}(y) dt\cr
&=&\frac{1}{p+1}\int_{L(P_{e_n^\perp}f)}|L_t(f)\cap(y+\langle e_n)\rangle|_1^{p+1}d\nu_{n}(y,t).
\end{eqnarray*}
\end{proof}

Let us now prove Theorem \ref{thm:DiscreteFunctionalZhang}. We follow the same lines as in the proof of \eqref{eq:InclusionFunctionalZhangRadialFunctions} provided in Section \ref{subsec:NewProof}.

\begin{proof}[Proof of Theorem \ref{thm:DiscreteFunctionalZhang}]
Let $f\in\mathcal{F}(\R^n)$ such that $\Vert f\Vert_\infty=f(0)$. Let $L(f)\subseteq\R^n\times[0,\infty)$ and $L_t(f)\subseteq\R^n$ for any $t\geq0$ be defined as in Section \ref{sec:LogconcaveFunctions}. Taking into account that $P_{e_n^\perp}(L(f))=L(P_{e_n^\perp}f)\subseteq e_n^\perp\times[0,\infty)$, we define the function $h:L(P_{e_n^\perp}f)\to[0,\infty)$ given by
$$
h(y,t)=\frac{1}{2}|L(f) \cap ((y,t)+\langle e_n \rangle )|_1=\frac{1}{2}|L_t(f) \cap (y+\langle e_n \rangle)|_1,
$$
for every $(y,t)\in\R^{n-1}\times[0,\infty)$ such that $y\in P_{e_n^\perp}(L_t(f))$. Since $L(f)$ is convex and $L_t(f)$ is a bounded convex set for any $t\geq0$, $h$ is well-defined and is a concave function. Moreover, the hypograph of $h$ is
\begin{eqnarray}\label{eq:HypographHOnProjection}
\textrm{hyp}(h)&=&\{(y,t,s)\in e_{n}^\perp\times[0,\infty)\times[0,\infty)\,:\, 0\leq s\leq h(y,t)\}\cr
&=&\{(y,t,s)\in e_{n}^\perp\times[0,\infty)\times[0,\infty)\,:\, (y,r)\in S_{e_n}(L_t(f)),\forall 0\leq r<s\}\cr
&=&\{(y,t,s)\in e_{n}^\perp\times[0,\infty)\times[0,\infty)\,:\,(y,r)\in L_t(S_{e_n}(f)),\forall 0\leq r<s\}\cr
&=&\{(y,t,s)\in e_{n}^\perp\times[0,\infty)\times[0,\infty)\,:\,(y,r,t)\in L(S_{e_n}(f)),\forall 0\leq r<s\}.
\end{eqnarray}

Therefore, we can apply the discrete functional version of Berwald's inequality, Theorem \ref{thm:DiscreteFunctionalBerwald} on $\tilde{L}=L(P_{e_n^\perp}f)=P_{e_n^\perp}(L(f))=P_{e_n^\perp}(L(S_{e_n}(f)))$ with parameters $p=1<n+1=q$, and obtain
$$
	\left(\frac{1}{\Gamma(n+2)\nu_n(\tilde{L})}\int_{\tilde{L}} h^{n+1}(y,t)d\nu_n(y,t)\right)^{1/(n+1)}\leq\frac{1}{\Gamma(2)\nu_n(\tilde{L})}\int_{\tilde{L}+C_{n-1}} (h\star\chi_{C_{n-1}})(y,t)d\nu_n(y,t).
$$

Equivalently,
$$
\frac{1}{n+1}\int_{\tilde{L}} h^{n+1}(y,t)d\nu_n(y,t)\leq \frac{n!}{\nu_n(\tilde{L})^n}\left(\int_{\tilde{L}+C_{n-1}} (h\star\chi_{C_{n-1}})(y,t)d\nu_n(y,t)\right)^{n+1}.
$$

On the one hand, by Lemma \ref{lem:IdentityDiscrete}
\begin{align*}
\frac{1}{n+1}\int_{\tilde{L}} h^{n+1}(y,t)d\nu_n(y,t)&=\frac{n}{2^{n+1}}\int_0^\infty r^{n-1}g_f^\mu(re_n)dr\cr
&=\frac{n}{2^{n+1}}\int_0^\infty r^{n-1}\int_{\R^n}\min\left\{\frac{f(z)}{\Vert f\Vert_\infty},\frac{f(z-re_n)}{\Vert f\Vert_\infty}\right\}d\mu_n(z)dr.
\end{align*}

On the other hand, since $\tilde{L}=L(P_{e_n^\perp}f)$, by \eqref{eq:nu(L)Integral} we have
$$
\nu_n(\tilde{L})=\int_{e_n^\perp}\frac{P_{e_n^\perp}f (y)}{\Vert P_{e_n^\perp}f\Vert_\infty}dG_{n-1}(y)=\int_{e_n^\perp}\frac{P_{e_n^\perp}f (y)}{\Vert f\Vert_\infty}dG_{n-1}(y).
$$

Finally, let us see that
\begin{equation}\label{eq:FromhToSymmetrization}
\int_{\tilde{L}+C_{n-1}} (h\star\chi_{C_{n-1}})(y,t)d\nu_n(y,t)\leq\frac{1}{2}\int_{\R^n}\frac{(S_{e_n}(f)\star\chi_{C_{n-1}})(x)}{\Vert f\Vert_\infty}d\mu_n(x).
\end{equation}

Taking into account \eqref{eq:HypographHOnProjection}, we have that for any $(y,t)\in \tilde{L}+C_{n-1}$ and any $s\geq0$
\begin{eqnarray*}
&&(h\star\chi_{C_{n-1}})(y,t)\geq s\Leftrightarrow\sup_{z\in C_{n-1}\atop y+z\in P_{e_n^\perp}L_t(f)}h(y+z,t)\geq s\cr
&\Leftrightarrow&\forall 0\leq r<s,\exists z\in C_{n-1}\textrm{ such that }y+z\in P_{e_n^\perp}L_t(f) \textrm{ and }h(y+z,t)> r\cr
&\Leftrightarrow&\forall 0\leq r<s,\exists z\in C_{n-1}\textrm{ such that }(y+z,r,t)\in L(S_{e_n}(f))\cr
&\Leftrightarrow& \forall 0\leq r<s, \sup_{z\in C_{n-1}} (S_{e_n}(f))(y+z,r)\geq e^{-t}\Vert f\Vert_\infty\cr
&\Leftrightarrow& \forall 0\leq r<s, (y,r,t)\in L(S_{e_n}(f)\star\chi_{C_{n-1}}) \cr
&\Leftrightarrow& \forall 0\leq r<s, (y,r)\in L_t(S_{e_n}(f)\star\chi_{C_{n-1}}).
\end{eqnarray*}

Therefore, for any $(y,t)\in \tilde{L}+C_{n-1}$ we have that
\begin{eqnarray*}
(h\star\chi_{C_{n-1}})(y,t)\geq s\Leftrightarrow\frac{1}{2}|L_t(S_{e_n}(f)\star\chi_{C_{n-1}})\cap (y+\langle e_n\rangle)|_1\geq s,
\end{eqnarray*}
and then
$$
(h\star\chi_{C_{n-1}})(y,t)=\frac{1}{2}|L_t(S_{e_n}(f)\star\chi_{C_{n-1}})\cap (y+\langle e_n\rangle)|_1,\quad\forall (y,t)\in \tilde{L}+C_{n-1}.
$$
Consequently, taking into account that for any $t\geq 0$,  we have by \eqref{eq:InclusionSuperlevelSetsAsplund} that
\begin{eqnarray*}
L_t(P_{e_n^\perp}f)+C_{n-1}&=&P_{e_n^\perp}(L_t(f)+C_{n-1})=P_{e_n^\perp}(L_t(S_{e_n}(f))+C_{n-1})\cr
&\subseteq& P_{e_n^\perp}(L_t(S_{e_n}(f)\star\chi_{C_{n-1}})).
\end{eqnarray*}
Therefore, using also \eqref{eq:sigma(L)Integral},
\begin{eqnarray*}
&&\int_{\tilde{L}+C_{n-1}} (h\star\chi_{C_{n-1}})(y,t)d\nu_n(y,t)\cr
&=&\frac{1}{2}\int_{\tilde{L}+C_{n-1}} |L_t(S_{e_n}(f)\star\chi_{C_{n-1}})\cap (y+\langle e_n\rangle)|_1d\nu_n(y,t)\cr
&=&\frac{1}{2}\int_0^\infty e^{-t}\int_{L_t(P_{e_n^\perp}f)+C_{n-1}}|L_t(S_{e_n}(f)\star\chi_{C_{n-1}})\cap (y+\langle e_n\rangle)|_1 dG_{n-1}(y)dt\cr
&\leq&\frac{1}{2}\int_0^\infty e^{-t}\int_{P_{e_n^\perp}(L_t(S_{e_n}(f)\star\chi_{C_{n-1}}))}|L_t(S_{e_n}(f)\star\chi_{C_{n-1}})\cap (y+\langle e_n\rangle)|_1 dG_{n-1}(y)dt\cr
&=&\frac{1}{2}\int_0^\infty e^{-t}\mu_n(L_t(S_{e_n}(f)\star\chi_{C_n}))dt=\frac{1}{2}\sigma_{n+1}(L(S_{e_n}(f)\star\chi_{C_{n-1}}))\cr
&=&\frac{1}{2}\int_{\R^n}\frac{(S_{e_n}(f)\star\chi_{C_{n-1}})(x)}{\Vert S_{e_n}(f)\star\chi_{C_{n-1}}\Vert_\infty}d\mu_n(x)=\frac{1}{2}\int_{\R^n}\frac{(S_{e_n}(f)\star\chi_{C_{n-1}})(x)}{\Vert f\Vert_\infty}d\mu_n(x).
\end{eqnarray*}
Thus, we obtain \eqref{eq:FromhToSymmetrization} and this completes the proof.
\end{proof}

\subsection{From the discrete to the continuous functional Zhang's inequality}

In this subsection we are going to prove that that the discrete functional version of Zhang's inequality given by Theorem \ref{thm:DiscreteFunctionalZhang} implies the functional version, \eqref{eq:InclusionFunctionalZhangRadialFunctions}.

\begin{proof}[Proof of Theorem \ref{thm:DiscreteFunctionalZhang} implies \eqref{eq:InclusionFunctionalZhangRadialFunctions}]

Let $f=\Vert f\Vert_{\infty} e^{-u}\in\mathcal{F}(\R^n)$. We can assume, without loss of generality, that $\Vert f\Vert_\infty=f(0)$. Since every $f\in\mathcal{F}(\R^n)$ is continuous on the interior of its support, $\overline{\nu}(\partial L(f))=0$, and taking into account \eqref{eq:barnu(L)Integral}, we can also assume, without loss of generality, that $L(f)$ is closed and $f$ is continuous on its support. Otherwise, change $f$ to $f_1=\Vert f\Vert_\infty e^{-u_1}$, the log-concave function such that $\textrm{epi}(u_1)=\overline{\textrm{epi}(u)}$. We define, for any $\lambda>0$, the function $f_\lambda \colon \R^n \to [0, \infty)$ given by $f_\lambda(x)=f\left( \frac{x}{\lambda} \right)$. Notice that $f_\lambda$ is log-concave, $\norma{f_\lambda}_{\infty}=\norma{f}_{\infty}=f(0)=f_\lambda(0)$ and for any $t \geq 0$, $L_t(f_\lambda)=\{ x \in \R^n : f_\lambda(x) \geq e^{-t} \norma{f_\lambda}_{\infty} \}=\lambda L_t(f).$ Applying Theorem \ref{thm:DiscreteFunctionalZhang} to the function $f_\lambda$ we have that for any $\lambda>0$
\begin{equation}\label{eq:ZhangfLambda}
n\int_0^\infty r^{n-1}\int_{\R^n}\min\{f_\lambda(z),f_\lambda(z-re_n)\}d\mu(z)dr\leq n!\frac{\left(\int_{\R^n}(S_{e_n}(f_\lambda)\star\chi_{C_{n-1}})(x)d\mu(x)\right)^{n+1}}{\left(\int_{e_n^\perp}P_{e_n^\perp}f_\lambda(y)dG_{n-1}(y)\right)^n}.
\end{equation}

Taking into account that $L(f)$ is closed, for every $t\geq0$ and every $\lambda>0$ we have that $L_t(S_{e_n} (f_\lambda)=S_{e_n}(L_t(f_\lambda))$. Using \eqref{eq:InclusionSuperlevelSetsAsplund} and \eqref{eq:InverseInclusionSuperLevelSetsAsplund}, as a consequence of \eqref{eq:SteinerSymmetrizationInclusion}, we have
\begin{eqnarray*}
\int_{\R^n}\frac{(S_{e_n}(f_\lambda)\star\chi_{C_{n-1}})(x)}{\Vert f\Vert_\infty}d\mu(x)&=&\int_0^{\infty} e^{-t} \mu(L_t(S_{e_n} (f_\lambda)\star\chi_{C_{n-1}})) dt \cr
&\leq&\int_0^{\infty} e^{-t} \mu(S_{e_n}(L_t(f_\lambda))+\overline{C_{n-1}}) dt \cr
&\leq& \int_0^{\infty} e^{-t} \mu(\lambda L_t(f) + \overline{C_{n-1}}) dt\cr
&=&\int_{\R^n}\frac{(f_\lambda\star\chi_{\overline{C_{n-1}}})(x)}{\Vert f\Vert_\infty}d\mu(x).
\end{eqnarray*}

Thus, dividing both sides of inequality \eqref{eq:ZhangfLambda} by $\lambda^{2n}=\frac{\lambda^{n(n+1)}}{\lambda^{n(n-1)}}$, we have
\begin{equation}\label{eq:first_application_zhang}
\frac{n}{\lambda^{2n}}\int_0^\infty r^{n-1}\int_{\R^n}\min\{f_\lambda(z),f_\lambda(z-re_n)\}d\mu(z)dr\leq n!\frac{\left(\frac{1}{\lambda^n}\int_{\R^n}(f_\lambda\star\chi_{\overline{C_{n-1}}})(x)d\mu(x)\right)^{n+1}}{\left(\frac{1}{\lambda^{n-1}}\int_{e_n^\perp}P_{e_n^\perp}f_\lambda(y)dG_{n-1}(y)\right)^n}.
\end{equation}

On the one hand, notice that for any $\lambda>0$ we have $P_{e_n^\perp}f_\lambda=(P_{e_n^\perp}f)_\lambda$. Therefore, by Lemma \ref{lem:LogConcaveIntegralDiscreteToContinuous} applied with $M=\{0\}$ we have
\begin{equation}\label{eq:limitRightHandSideDenominator}
\lim_{\lambda\to\infty}\frac{1}{\lambda^{n-1}}\int_{e_n^\perp}P_{e_n^\perp}f_\lambda(y)dG_{n-1}(y)=\int_{e_n^\perp}P_{e_n^\perp}f(y)dy.
\end{equation}

On the other hand, notice that, by Lemma \ref{lem:LogConcaveIntegralMuToContinuous} applied with $M=\overline{C_{n-1}}$, we have
\begin{equation}\label{eq:limitRightHandSideNumerator}
\lim_{\lambda\to\infty}\frac{1}{\lambda^n}\int_{\R^n}(f_\lambda\star\chi_{\overline{C_{n-1}}})(x)d\mu(x)=\int_{\R^n}f(x)dx.
\end{equation}

Let us prove that
\begin{eqnarray}\label{eq:LimitLeftHandSide}
&&\lim_{\lambda\to\infty}\frac{1}{\lambda^{2n}}\int_0^\infty r^{n-1}\int_{\R^n}\min\{f_\lambda(z),f_\lambda(z-re_n)\}d\mu(z)dr\cr
&=&\int_0^\infty r^{n-1}\int_{\R^n}\min\{f(z),f(z-re_n)\}dzdr.
\end{eqnarray}
By \eqref{eq:integralMuandm}, it is enough to prove that
\begin{eqnarray}\label{eq:LimitLeftHandSideFirstTerm}
&&\lim_{\lambda\to\infty}\frac{1}{\lambda^{2n}}\int_0^\infty r^{n-1}\int_{\R^n}\min\left\{\frac{f_\lambda(z)}{\Vert f\Vert_\infty},\frac{f_\lambda(z-re_n)}{\Vert f\Vert_\infty}\right\}dG_n(z)dr\cr
&=&\int_0^\infty r^{n-1}\int_{\R^n}\min\left\{\frac{f(z)}{\Vert f\Vert_\infty},\frac{f(z-re_n)}{\Vert f\Vert_\infty}\right\}dzdr
\end{eqnarray}
and
\begin{equation}\label{eq:LimitLeftHandSideSecondTerm}
\lim_{\lambda\to\infty}\frac{1}{\lambda^{2n}}\int_0^\infty r^{n-1}\int_{e_n^\perp}P_{e_n^\perp}\min\left\{\frac{f_\lambda(\cdot)}{\Vert f\Vert_\infty},\frac{f_\lambda(\cdot-re_n)}{\Vert f\Vert_\infty}\right\}(y)dG_{n-1}(y)dr=0.
\end{equation}

Notice that for every $\lambda>0$ we have
\begin{eqnarray*}
&&\frac{1}{\lambda^{2n}}\int_0^\infty r^{n-1}\int_{\R^n}\min\left\{\frac{f_\lambda(z)}{\Vert f\Vert_\infty},\frac{f_\lambda(z-re_n)}{\Vert f\Vert_\infty}\right\}dG_n(z)dr\cr
&=&\frac{1}{\lambda^{2n}}\int_0^\infty r^{n-1}\int_0^\infty e^{-t} G_n\left(\lambda L_t(f)\cap(re_n+\lambda L_t(f))\right)dtdr\cr
&=&\frac{1}{\lambda^{2n}}\int_0^\infty \int_0^{\rho_{\lambda(L_t(f)-L_t(f))}(e_n)}e^{-t}r^{n-1}G_n\left(\lambda\left(L_t(f)\cap\left(\frac{r}{\lambda}e_n+L_t(f)\right)\right)\right)drdt\cr
&=&\int_0^\infty \int_0^{\rho_{L_t(f)-L_t(f)}(e_n)}e^{-t}s^{n-1}\frac{G_n\left(\lambda\left(L_t(f)\cap(se_n+L_t(f))\right)\right)}{\lambda^n}dsdt.\cr
\end{eqnarray*}

On the one hand, by \eqref{discrete_continuous_volume}, for every $t\geq0$ and every $s\in[0,\rho_{L_t(f)-L_t(f)}(e_n)]$, we have that
$$
\lim_{\lambda\to\infty}e^{-t}s^{n-1}\frac{G_n\left(\lambda\left(L_t(f)\cap(se_n+L_t(f))\right)\right)}{\lambda^n}=e^{-t}s^{n-1}|L_t(f)\cap (se_n+L_t(f))|.
$$

On the other hand, for every $t\geq0$, every $s\in[0,\rho_{L_t(f)-L_t(f)}(e_n)]$ and every $\lambda>\frac{1}{2}$, by \eqref{eq:LatticePointAndVolume} we have
\begin{eqnarray*}
e^{-t}s^{n-1}\frac{G_n\left(\lambda\left(L_t(f)\cap(se_n+L_t(f))\right)\right)}{\lambda^n}&\leq&e^{-t}s^{n-1}\frac{\left|\lambda\left(L_t(f)\cap(se_n+L_t(f))\right)+\frac{1}{2}B_\infty^n\right|}{\lambda^n}\cr
&\leq&e^{-t}s^{n-1}\frac{\left|\lambda\left(L_t(f)\cap(se_n+L_t(f))\right)+\lambda B_\infty^n\right|}{\lambda^n}\cr
&=&e^{-t}s^{n-1}\left|\left(L_t(f)\cap(se_n+L_t(f))\right)+B_\infty^n\right|\cr
&\leq&e^{-t}s^{n-1}\left|(L_t(f)+B_\infty^n)\cap(se_n+L_t(f)+B_\infty^n)\right|\cr
&=&e^{-t}s^{n-1}\left|L_t(f\star\chi_{B_\infty^n})\cap(se_n+L_t(f\star\chi_{B_\infty^n}))\right|\cr
&\leq&e^{-t}s^{n-1}\left|L_t(f\star\chi_{B_\infty^n})\right|.
\end{eqnarray*}
Let us see that this upper bound is integrable on $\{(t,s)\in[0,\infty)^2\,:\,0\leq s\leq \rho_{L_t(f)-L_t(f)}(e_n)\}$. Since $\frac{f\star\chi_{B_\infty^n}}{\Vert f\Vert_\infty}\in\mathcal{F}(\R^n)$, there exist $A,B>0$ such that $f\star\chi_{B_\infty^n} (x)\leq \Vert f\Vert_\infty A e^{-B\Vert x\Vert_2}$ for every $x\in\R^n$. Therefore, for every $t\geq0$,
\begin{equation}\label{eq:inclusionInBall}
L_t(f\star\chi_{B_\infty^n})\subseteq\frac{t+\log A}{B} B_2^n.
\end{equation}
Thus, we have that
$$
\left|L_t(f\star\chi_{B_\infty^n})\right|\leq\frac{(t+\log A)^{n}}{B^{n}}|B_2^{n}|
$$
and
\begin{equation}\label{eq:InclusionDifferenceBodyInBall}
L_t(f)-L_t(f)\subseteq \frac{2(t+\log A)}{B} B_2^n.
\end{equation}
Therefore,
\begin{eqnarray*}
&&\int_0^\infty \int_0^{\rho_{L_t(f)-L_t(f)}(e_n)}e^{-t}s^{n-1}\left|L_t(f\star\chi_{B_\infty^n})\right|dsdt\cr
&\leq&\int_0^\infty\int_0^{\frac{2(t+\log A)}{B}}e^{-t}s^{n-1}\frac{(t+\log A)^{n}}{B^{n}}|B_2^{n}|dsdt\cr
&=&\frac{2^n|B_2^n|}{nB^{2n}}\int_0^\infty e^{-t}(t+\log A)^{2n}dt,
\end{eqnarray*}
which is finite. This proves \eqref{eq:LimitLeftHandSideFirstTerm}. Let us now prove \eqref{eq:LimitLeftHandSideSecondTerm}. Notice that for every $\lambda>0$ we have
\begin{eqnarray*}
&&\frac{1}{\lambda^{2n-2}}\int_0^\infty r^{n-1}\int_{e_n^\perp}P_{e_n^\perp}\min\left\{\frac{f_\lambda(\cdot)}{\Vert f\Vert_\infty},\frac{f_\lambda(\cdot-re_n)}{\Vert f\Vert_\infty}\right\}(y)dG_{n-1}(y)dr\cr
&=&\frac{1}{\lambda^{2n-2}}\int_0^\infty r^{n-1}\int_0^\infty e^{-t} G_{n-1}\left(P_{e_n^\perp}\left(\lambda L_t(f)\cap\left(r e_n+\lambda L_t(f)\right)\right)\right)dtdr\cr
&=&\frac{1}{\lambda^{2n-2}}\int_0^\infty r^{n-1}\int_0^\infty e^{-t} G_{n-1}\left(\lambda P_{e_n^\perp}\left(L_t(f)\cap\left(\frac{r}{\lambda} e_n+ L_t(f)\right)\right)\right)dtdr\cr
&=&\int_0^\infty s^{n-1}\int_0^\infty e^{-t} \frac{G_{n-1}\left(\lambda P_{e_n^\perp}\left(L_t(f)\cap\left(se_n+ L_t(f)\right)\right)\right)}{\lambda^{n-1}}dtds\cr
&=&\int_0^\infty \int_0^{\rho_{L_t(f)-L_t(f)}(e_n)}e^{-t}s^{n-1} \frac{G_{n-1}\left(\lambda P_{e_n^\perp}\left(L_t(f)\cap\left(se_n+ L_t(f)\right)\right)\right)}{\lambda^{n-1}}dsdt.\cr
\end{eqnarray*}

On the one hand, by \eqref{discrete_continuous_volume}, we have that for every $t\geq0$ and every $s\in[0,\rho_{L_t(f)-L_t(f)}(e_n)]$,
$$
\lim_{\lambda\to\infty}e^{-t}s^{n-1}\frac{G_{n-1}\left(\lambda P_{e_n^\perp}\left(L_t(f)\cap\left(se_n+ L_t(f)\right)\right)\right)}{\lambda^{n-1}}=e^{-t}s^{n-1}|P_{e_n^\perp}(L_t(f)\cap (se_n+L_t(f)))|.
$$
On the other hand, for every $t\geq0$ and every $s\in[0,\rho_{L_t(f)-L_t(f)}(e_n)]$ and every $\lambda>\frac{1}{2}$, by \eqref{eq:LatticePointAndVolume} we have
\begin{eqnarray*}
e^{-t}s^{n-1}\frac{G_{n-1}\left(\lambda P_{e_n^\perp}\left(L_t(f)\cap\left(se_n+ L_t(f)\right)\right)\right)}{\lambda^{n-1}}&\leq&e^{-t}s^{n-1}\frac{\left|\lambda P_{e_n^\perp}\left(L_t(f)\cap\left(se_n+ L_t(f)\right)\right)+\frac{1}{2}P_{e_n^\perp}(B_\infty^n)\right|}{\lambda^{n-1}}\cr
&\leq&e^{-t}s^{n-1}\frac{\left|\lambda P_{e_n^\perp}\left(L_t(f)\cap\left(se_n+ L_t(f)\right)\right)+\lambda P_{e_n^\perp}(B_\infty^n)\right|}{\lambda^{n-1}}\cr
&=&e^{-t}s^{n-1}\left|P_{e_n^\perp}\left(L_t(f)\cap\left(se_n+ L_t(f)\right)\right)+P_{e_n^\perp}(B_\infty^n)\right|\cr
&=&e^{-t}s^{n-1}\left|P_{e_n^\perp}\left(L_t(f)\cap\left(se_n+ L_t(f)\right)+B_\infty^n\right)\right|\cr
&\leq&e^{-t}s^{n-1}\left|P_{e_n^\perp}\left((L_t(f)+B_\infty^n)\cap\left(se_n+ L_t(f)+B_\infty^n\right)\right)\right|\cr
&\leq&e^{-t}s^{n-1}\left|P_{e_n^\perp}\left(L_t(f\star\chi_{B_\infty^n})\cap\left(se_n+ L_t(f\star\chi_{B_\infty^n})\right)\right)\right|\cr
&\leq&e^{-t}s^{n-1}\left|P_{e_n^\perp}\left(L_t(f\star\chi_{B_\infty^n})\right)\right|.\cr
\end{eqnarray*}

By \eqref{eq:inclusionInBall}, we have that
$$
\left|P_{e_n^\perp}\left(L_t(f\star\chi_{B_\infty^n})\right)\right|\leq\frac{(t+\log A)^{n-1}}{B^{n-1}}|B_2^{n-1}|
$$
and, taking also into account \eqref{eq:InclusionDifferenceBodyInBall},
\begin{eqnarray*}
&&\int_0^\infty \int_0^{\rho_{L_t(f)-L_t(f)}(e_n)}e^{-t}s^{n-1}\left|P_{e_n^\perp}\left(L_t(f\star\chi_{B_\infty^n})\right)\right|dsdt\cr
&\leq&\int_0^\infty \int_0^{\frac{2(t+\log A)}{B}}e^{-t}s^{n-1}\frac{(t+\log A)^{n-1}}{B^{n-1}}|B_2^{n-1}|dsdt\cr
&=&\frac{|B_2^{n-1}|}{n B^{2n-1}}\int_0^\infty e^{-t}2^n(t+\log A)^{2n-1}dt,
\end{eqnarray*}
which is finite. By the dominated convergence theorem we obtain
\begin{eqnarray*}
&&\lim_{\lambda\to\infty}\frac{1}{\lambda^{2n-2}}\int_0^\infty r^{n-1}\int_{e_n^\perp}P_{e_n^\perp}\min\left\{\frac{f_\lambda(\cdot)}{\Vert f\Vert_\infty},\frac{f_\lambda(\cdot-re_n)}{\Vert f\Vert_\infty}\right\}(y)dG_{n-1}(y)dr\cr
&=&\lim_{\lambda\to\infty}\int_0^\infty \int_0^{\rho_{L_t(f)-L_t(f)}(e_n)}e^{-t}s^{n-1} \frac{G_{n-1}\left(\lambda P_{e_n^\perp}\left(L_t(f)\cap\left(se_n+ L_t(f)\right)\right)\right)}{\lambda^{n-1}}dtds\cr
&=&\int_0^\infty \int_0^{\rho_{L_t(f)-L_t(f)}(e_n)}e^{-t}s^{n-1}|P_{e_n^\perp}(L_t(f)\cap (se_n+L_t(f)))|dsdt,
\end{eqnarray*}
and then
$$
\lim_{\lambda\to\infty}\frac{1}{\lambda^{2n}}\int_0^\infty r^{n-1}\int_{e_n^\perp}P_{e_n^\perp}\min\left\{\frac{f_\lambda(\cdot)}{\Vert f\Vert_\infty},\frac{f_\lambda(\cdot-re_n)}{\Vert f\Vert_\infty}\right\}(y)dG_{n-1}(y)dr=0.
$$
This proves \eqref{eq:LimitLeftHandSideSecondTerm} and, consequently, \eqref{eq:LimitLeftHandSide}. Therefore, taking limit as $\lambda\to\infty$ in \eqref{eq:first_application_zhang} and using \eqref{eq:limitRightHandSideDenominator},\eqref{eq:limitRightHandSideNumerator} and \eqref{eq:LimitLeftHandSide} we obtain \eqref{eq:InclusionFunctionalZhangRadialFunctions} as a consequence of Theorem \ref{thm:DiscreteFunctionalZhang}.

\end{proof}

\section{Discrete functional Rogers-Shephard inequality}\label{sec:DiscreteFunctionalRogersShephard}

In this section we are going to prove Theorem \ref{thm:DiscreteFunctionalRogersShephard}. We will obtain it as a consequence of \eqref{eq:Rogers-Shephard discrete} and we will see that, in fact, it is equivalent to \eqref{eq:Rogers-Shephard discrete}. We will also see that Theorem \ref{thm:DiscreteFunctionalRogersShephard} implies \eqref{eq:FunctionalRogersSephard}.

\begin{proof}[Proof of Theorem \ref{thm:DiscreteFunctionalRogersShephard}]
Let $f\in\mathcal{F}(\R^n)$. First of all, notice that for any $x,z\in\R^n$ we have
$$
\min\left\{\frac{f(z)}{\Vert f\Vert_\infty},\frac{f(z-x)}{\Vert f\Vert_\infty}\right\}\geq\frac{f(z)f(z-x)}{\Vert f\Vert_\infty^2},
$$
since both $\frac{f(z)}{\Vert f\Vert_\infty},\frac{f(z-x)}{\Vert f\Vert_\infty}\in [0,1]$. Notice also that for any $t\geq0$, if $x\in L_t(f\star\tilde{f})$, we have that $\displaystyle{\sup_{z\in\R^n}f(z)f(z-x)\geq e^{-t}\Vert f\Vert_\infty^2}$ and then, there exists a sequence $(z_n)_{n=1}^{\infty}\subseteq\R^n$ such that
$$
\lim_{n\to\infty}\frac{f(z_n)f(z_n-x)}{\Vert f\Vert_\infty^2}\geq e^{-t}.
$$
Besides, fixing any $\varepsilon>0$ we can assume, that for every $n\in\N$
$$
\min\left\{\frac{f(z_n)}{\Vert f\Vert_\infty},\frac{f(z_n-x)}{\Vert f\Vert_\infty}\right\}\geq\frac{f(z_n)f(z_n-x)}{\Vert f\Vert_\infty^2}>e^{-(t+\varepsilon)}.
$$
Therefore, for any fixed $\varepsilon>0$ we can assume that, for every $n\in\N$, we have that $z_n\in L_{t+\varepsilon}(f)\cap(x+ L_{t+\varepsilon}(f))$ and then, for any $\varepsilon>0$ we have that $L_{t+\varepsilon}(f)\cap(x+ L_{t+\varepsilon}(f))\neq\emptyset$ and then $x\in L_{t+\varepsilon}(f)-L_{t+\varepsilon}(f)$ for every $\varepsilon>0$.



Thus, for every $\varepsilon>0$
\begin{equation}\label{eq:InclusionSuperlevelSetsRS}
L_t(f\star\tilde{f})\subseteq  L_{t+\varepsilon}(f)-L_{t+\varepsilon}(f).
\end{equation}

Therefore, by \eqref{eq:Rogers-Shephard discrete}, we have that for every $t\geq0$ and every $\varepsilon>0$
$$
G_n(L_t(f\star\tilde{f}))\leq G_n\left(L_{t+\varepsilon}(f)-L_{t+\varepsilon}(f)\right)\leq\binom{2n}{n}G_n\left(L_{t+\varepsilon}(f)+\frac{3}{4}C_n\right).
$$
By \eqref{eq:InclusionSuperlevelSetsAsplund} we have
$$
L_{t+\varepsilon}(f)+\frac{3}{4}C_n\subseteq L_{t+\varepsilon}\left(f\star \chi_{\frac{3}{4}C_n} \right).
$$
Therefore, for every $t\geq0$ and every $\varepsilon>0$,
$$
G_n(L_t(f\star\tilde{f}))\leq \binom{2n}{n}G_n\left(L_{t+\varepsilon}\left(f\star \chi_{\frac{3}{4}C_n} \right)\right).
$$
Since for every $t\geq 0$ this is true for every $\varepsilon>0$, we have that for any $t\geq 0$
$$
G_n(L_t(f\star\tilde{f}))\leq \binom{2n}{n}G_n\left(L_{t}\left(f\star \chi_{\frac{3}{4}C_n} \right)\right)
$$
and then
$$
\int_0^\infty e^{-t}G_n(L_t(f\star\tilde{f}))dt\leq\binom{2n}{n}\int_0^\infty e^{-t}G_n\left(L_t\left(f\star\chi_{\frac{3}{4}C_n} \right)\right)dt.
$$

On the one hand, by \eqref{eq:nu(L)Integral},
$$
\int_0^\infty e^{-t}G_n(L_t(f\star\tilde{f}))dt=\nu(L(f\star\tilde{f}))=\frac{1}{\Vert f\Vert_\infty^2}\int_{\R^n}f\star\tilde{f}(x)dG_n(x).
$$

On the other hand,  also by \eqref{eq:nu(L)Integral},
$$
\int_0^\infty e^{-t}G_n\left(L_t\left(f\star\chi_{\frac{3}{4}C_n}\right)\right)dt=\nu\left(L\left(f\star\chi_{\frac{3}{4}C_n}\right)\right)=\frac{1}{\Vert f\Vert_\infty}\int_{\R^n}f\star\chi_{\frac{3}{4}C_n}(x)dx.
$$
Thus, we obtain
$$
\int_{\R^n} f \star \tilde{f}(x) dG_n(x) \leq \binom{2n}{n} \norma{f}_{\infty} \int_{\R^n} f\star\chi_{\frac{3}{4}C_n}(x) dG_n(x).
$$
\end{proof}

Applying Theorem \ref{thm:DiscreteFunctionalRogersShephard} to the characteristic function of any bounded convex set with non-empty interior, we obtain \eqref{eq:Rogers-Shephard discrete} and, therefore, they are equivalent. We state it in the following corollary:

\begin{cor}
Let $K\subseteq\R^n$ be a bounded convex set with non-empty interor. Then
$$
G_n(K-K) \leq \binom{2n}{n} G_n \left( K+ \frac{3}{4}C_n \right).
$$
\end{cor}

\begin{proof}
Let us consider $f=\chi_K\in\mathcal{F}(\R^n)$. Then $\Vert f\Vert_\infty=1$, $\tilde{f}=\chi_{-K}$ and the functions $f\star\tilde{f}$ and $f\star\chi_{\frac{3}{4}C_n}$ are given by
$$
f\star\tilde{f}=\chi_{K-K}\hspace{1cm}\textrm{ and }\hspace{1cm}f\star\chi_{\frac{3}{4}C_n}=\chi_{K+\frac{3}{4}C_n}.
$$
Thus, Theorem \ref{thm:DiscreteFunctionalRogersShephard} gives
$$
G_n(K-K) \leq \binom{2n}{n} G_n \left(K+ \frac{3}{4}C_n \right).
$$
\end{proof}

Let us point out that, in the same way as Theorem \ref{thm:DiscreteFunctionalRogersShephard} is equivalent to \eqref{eq:Rogers-Shephard discrete}, the continuous functional version of Rogers-Shephard inequality, \eqref{eq:FunctionalRogersSephard}, is equivalent to the geometric Rogers-Shephard inequality \eqref{eq:RogersShephard} (see the proof in \cite[Thm. 2.2]{AGJV1}, which is obtained in the same way from the same inclusion relation \eqref{eq:InclusionSuperlevelSetsRS}, by taking volumes instead of the lattice point enumerator measure, applying \eqref{eq:RogersShephard}, and integrating in $t$). Therefore, since the discrete version of Rogers-Shephard inequality, \eqref{eq:Rogers-Shephard discrete}, implies the continuous version \eqref{eq:RogersShephard}, we have that Theorem \ref{thm:DiscreteFunctionalRogersShephard} implies \eqref{eq:FunctionalRogersSephard}.

Nevertheless, we are going to give a direct proof of such implication:


\begin{proof}[Proof of Theorem \ref{thm:DiscreteFunctionalRogersShephard} implies \eqref{eq:FunctionalRogersSephard}]

Let $f=\Vert f\Vert_\infty e^{-u}\in\mathcal{F}(\R^n)$. Since every $f\in\mathcal{F}(\R^n)$ is continuous on the interior of its support, $\overline{\nu}(\partial L(f))=0$, and taking into account \eqref{eq:barnu(L)Integral}, we can assume without loss of generality that $L(f)$ is closed and $f$ is continuous on its support. Otherwise, change $f$ to $f_1=\Vert f\Vert_\infty e^{-u_1}$, the log-concave function such that $\textrm{epi}(u_1)=\overline{\textrm{epi}(u)}$ which we will rename as $f$.

Let us take $\lambda>0$ and let $f_{\lambda}\in\mathcal{F}(\R^n)$ be the function given by $f_{\lambda}(x)=f\left( \frac{x}{\lambda} \right)$. Notice that we have that $\widetilde{(f_{\lambda})}=(\tilde{f})_{\lambda}$, which we will denote $\Tilde{f_{\lambda}}$. Applying Theorem \ref{thm:DiscreteFunctionalRogersShephard} and taking int account \eqref{eq:IdentityLambda}, we obtain
\begin{equation*}
\int_{\R^n} (f\star\tilde{f})_\lambda (x) dG_n(x)=\int_{\R^n} f_{\lambda} \star \Tilde{f_{\lambda}} (x) dG_n(x) \leq \binom{2n}{n} \norma{f}_{\infty} \int_{\R^n} f_{\lambda} \star \chi_{\frac{3}{4}C_n} (x) dG_n(x).
\end{equation*}

Dividing both terms of the inequality by $\lambda^n$, we have that for any $\lambda>0$
\begin{equation} \label{eq:Rogers_Shephard_f_r}
\frac{1}{\lambda^n} \int_{\R^n} (f\star\tilde{f})_\lambda (x) dG_n(x) \leq \binom{2n}{n} \norma{f}_{\infty} \frac{1}{\lambda^n}  \int_{\R^n} f_{\lambda} \star \chi_{\frac{3}{4}C_n} (x) dG_n(x).
\end{equation}

On the one hand, by Lemma \ref{lem:LogConcaveIntegralDiscreteToContinuous} applied with $M=\{0\}$, we have
$$
\lim_{\lambda \to \infty} \frac{1}{\lambda^n}\int_{\R^n} (f\star\tilde{f})_\lambda (x) dG_n(x)= \int_{\R^n} f \star \tilde{f}(x) dx.
$$
On the other hand, Lemma \ref{lem:LogConcaveIntegralDiscreteToContinuous} applied with $M=\frac{3}{4}C_n$ gives
$$
\lim_{n\to\infty}\frac{1}{\lambda^n}  \int_{\R^n} f_{\lambda} \star \chi_{\frac{3}{4}C_n} (x) dG_n(x)=\int_{\R^n}f(x)dx,
$$
and, taking limits as $\lambda\to\infty$ in \eqref{eq:Rogers_Shephard_f_r}, we obtain \eqref{eq:FunctionalRogersSephard}.

\end{proof}

\end{document}